\documentclass[a4paper,12pt]{amsart}
\usepackage{graphicx, amsfonts} 

\usepackage{amsmath}
\usepackage{amsthm}
\usepackage{amssymb}
\usepackage{amscd}
\usepackage{color}
\usepackage[linktocpage]{hyperref}

\usepackage{comment}

\usepackage[T1]{fontenc}
\usepackage{libertine}

\usepackage{tikz}
\usetikzlibrary{intersections,calc,arrows.meta,cd}

\usepackage{here}

\newtheorem{thm}{Theorem}[section]
\newtheorem{cor}[thm]{Corollary}
\newtheorem{lem}[thm]{Lemma}
\newtheorem{prop}[thm]{Proposition}

\theoremstyle{definition}
\newtheorem{defn}[thm]{Definition}
\newtheorem{exam}[thm]{Example}
\newtheorem{conven}[thm]{Convention}
\newtheorem{rem}[thm]{Remark}
\newtheorem{conj}[thm]{Conjecture}

\numberwithin{equation}{section}

\def\NN{\mathbb N}
\def\ZZ{\mathbb Z}
\def\QQ{\mathbb Q}

\def\cS{\mathcal{S}}
\def\cT{\mathcal{T}}
\def\cR{\mathcal{R}}
\def\flS{\mathcal{S}^\flat}
\def\flR{\mathcal{R}^\flat}
\def\shS{\mathcal{S}^\sharp}
\def\shR{\mathcal{R}^\sharp}
\def\cU{\mathcal{U}}
\def\cV{\mathcal{V}}
\def\cW{\mathcal{W}}

\def\bb{\mathbf b}
\def\0{\mathbf 0}

\def\PSL{\operatorname{PSL}}

\def\Tr{\operatorname{tr}}

\def\cl{\mathsf{cl}}

\def\<{\langle}
\def\>{\rangle}

\def\Oguz{Kantarcı O\u{g}uz }

\title[$q$-Transposes of matrices and $q$-deformed rationals]{Transposes in the $q$-deformed modular group and their applications to $q$-deformed rational numbers}

\author[X. Ren]{Xin Ren}
\address{Osaka Central Advanced Mathematical Institute (OCAMI), Osaka Metropolitan University, 3-3-138 Sugimoto, Sumiyoshi-ku Osaka 558-8585, Japan}
\email{xin-ren@omu.ac.jp, xinren1213@gmail.com}

\author[K. Yanagawa]{Kohji Yanagawa} 
\address{Department of Mathematics, 
Kansai University, Osaka, 564-8680, Japan.}
\email{yanagawa@kansai-u.ac.jp}


\begin{document}

\begin{abstract}
The (right) $q$-deformed rational numbers was introduced by Morier-Genoud and Ovsienko, and its left variant, whose numerators and denominators are essentially the normalized Jones polynomials of rational links, by Bapat, Becker and Licata. These notions are based on continued fractions and the $q$-deformed modular group $\PSL_q(2,\ZZ)$-actions. In this paper, we introduce the \textit{$q$-transpose} 
for matrices in $\PSL_q(2,\ZZ)$ to refine the basic perspective of the theory. For example, we present a new proof and a refinement of a theorem of Leclere and Morier-Genoud stating that the trace of $A \in \PSL(2,\ZZ)$ is always palindromic and sign coherent. 
We also show arithmetic/combinatorial results on left $q$-deformed rationals (e.g., the criterion for their  palindromicity). Finally, we discuss the connection to the conjecture of Kantarcı O\u{g}uz on circular fence posets. 
\end{abstract}

\maketitle

\section{Introduction}
For an irreducible fraction $\frac{r}{s}$ and a formal parameter $q,$ the \textit{right $q$-deformed rational number} $\left[\displaystyle\frac{r}{s}\right]^\sharp_q$ is defined as a rational function in $\ZZ(q)$. This notion was originally introduced by Morier-Genoud and Ovsienko \cite{MO1, MOV} via continued fractions and the $q$-deformed modular group $\mathrm{PSL}_q(2,\ZZ)$-action. 
Here $\mathrm{PSL}_q(2,\ZZ)$ is a quotient group of a subgroup of $\mathrm{GL}(2,\ZZ[q^{\pm 1}])$. 
They further extended this notion to arbitrary real numbers \cite{MO2}, and Gaussian integers \cite{O}.  
Inspired by their works, Bapat, Becker and Licata \cite{BBL} gave another $q$-deformation $\left[\displaystyle\frac{r}{s}\right]^\flat_q$ of $\frac{r}s$, which is called the \textit{left $q$-deformed rational number}. These notions are related to many directions including 2-Calabi-Yau categories~\cite{BBL}, the Markov-Hurwitz approximation theory \cite{ SM, LMOV, XR}, the modular group \cite{LM,MOV}, Jones polynomial of rational knots \cite{KW, KR, MO1, BBL, XR2} and combinatoris of posets \cite{TBEC1,  OR, KO25}. 

For a matrix $A \in \mathrm{PSL}_q(2,\ZZ)$, we define two operations the {\it $q$-transpose} $A^{T_q}$ and the {\it orthogonal $q$-transpose} $A^{O_q}$ as follows. Let 
$$
\begin{pmatrix}
\cR(q) & \mathcal{V}(q) \\
\cS(q) &  \mathcal{U}(q) \\
\end{pmatrix}$$
be a matrix whose entries are elements in $\ZZ[q^{\pm 1}]$, we define
$$
A^{T_q}:=\begin{pmatrix}
\cR(q) & q^{-1}\cS(q) \\
q \mathcal{V}(q) &  \mathcal{U}(q) \\
\end{pmatrix}
\quad 
\text{and} \quad 
A^{O_q}:=\begin{pmatrix}
\mathcal{U}(q^{-1}) & q^{-1}\mathcal{V}(q^{-1}) \\
q \cS(q^{-1}) &  \cR(q^{-1}) \\
\end{pmatrix}.
$$
It can be checked that  $A \in \mathrm{PSL}_q(2,\ZZ)$, if and only if $A^{T_q} \in \mathrm{PSL}_q(2,\ZZ)$, if and only if $A^{O_q} \in \mathrm{PSL}_q(2,\ZZ)$ (see Lemmas~\ref{q-transpose}~and~\ref{orth & double q-transpose}). 

As for applications of these operations, we give some new proofs of existing results. For a polynomial $f(q)\in\mathbb{Z}[q]$, we write its \textit{reciprocal polynomial} by $f(q)^{\vee}$, that is, 
\[ f(q)^{\vee}:=q^{\mathsf{deg}(f)}f(q^{-1}).\] Moreover, we say $f(q)$ is  {\it palindrmoic} if $f(q)=f(q)^{\vee}.$ 
A theorem of  Leclere and Morier-Genoud (\cite[Theorem~3]{LM}) states that, for $A \in \PSL_q(2,\ZZ)$, the trace $\Tr A$ is a palindromic polynomial with non-negative coefficients after $\pm q^{N}$ times for some integer $N$. We give a new proof of this fact in Theorem~\ref{Tr A}, and a refinement in Proposition~\ref{3 types for trace}.

Some arithmetic properties on the right $q$-deformed rational numbers were recently studied in \cite{KMRWY}. 
In this paper, we also give their left versions as follows. For any irreducible fraction $\frac{r}{s},$ we write its left and right $q$-deformed rational numbers as 
$$
\left[\displaystyle\frac{r}{s}\right]^\flat_q=\frac{\flR_{\frac{r}{s}}(q)}{\flS_{\frac{r}{s}}(q)}, \qquad  \left[\displaystyle\frac{r}{s}\right]^\sharp_q=\frac{\shR_{\frac{r}{s}}(q)}{\shS_{\frac{r}{s}}(q)},$$
respectively, where
$\flR_{\frac{r}{s}}(q), \  \shR_{\frac{r}{s}}(q) \in \mathbb{Z}[q^{\pm 1}]$, $\flS_{\frac{r}{s}}(q), \ \shS_{\frac{r}{s}}(q)  \in \mathbb{Z}[q],$
$\flR_{\frac{r}{s}}(1)=\shR_{\frac{r}{s}}(1)=r$, and  $\flS_{\frac{r}{s}}(1)=\shS_{\frac{r}{s}}(1)=s.$
While we only treat the denominator $\flS_{\frac{r}{s}}(q)$ below, the similar hold for the numerator $\flR_{\frac{r}{s}}(q)$. 

\begin{thm}[see Theorem~\ref{arithmetic of flat}]
For irreducible fractions $\frac{r}s, \frac{r'}s$ with $rr' \equiv 1 \pmod{s}$ {\rm (}resp. $rr' \equiv -1 \pmod{s}${\rm )},  we have $\flS_{\frac{r}s}(q)=\flS_{\frac{r'}{s}}(q)$ {\rm (}resp. $\flS_{\frac{r}s}(q)=\flS_{\frac{r'}{s}}(q)^\vee${\rm )}.
\end{thm}

\begin{thm}[see Theorem~\ref{palin}]
For an irreducible fraction $\frac{r}s$, $\flS_{\frac{r}s}(q)$ is palindromic if and only if $r^2 \equiv -1 \pmod{s}$. 
\end{thm}

In \cite{KMRWY}, we showed that $\shS_{\frac{r}s}(q)=\shS_{\frac{r'}{s}}(q)$ {\rm (}resp. $\shS_{\frac{r}s}(q)=\shS_{\frac{r'}{s}}(q)^\vee${\rm )} if $rr' \equiv -1 \pmod{s}$ {\rm (}resp. $rr' \equiv 1 \pmod{s}${\rm )}, and $\shS_{\frac{r}s}(q)$ is palindromic if and only if $r^2 \equiv 1 \pmod{s}.$ So, the left and right $q$-deformed ration numbers behave oppositely in this sense. 

An irreducible fraction $\alpha =\frac{r}s >1$ determines a rational link $L(\alpha)$ via the continued fraction. We consider the {\it normalized Jones polynomial} $J_{\alpha}(q)\in \ZZ[q]$ of $L(\alpha),$ which can be computed by the left and right $q$-deformed rational numbers \cite{KR,MO1, BBL, XR2}. The simplest expression is $J_{\alpha}(q)=\flR_\alpha(q)^\vee$, and hence $J_{\alpha}(q)$ is palindromic if and only if $r^2 \equiv -1 \pmod{s}$.  
When $J_{\alpha}(q)$ is {\it not} palindromic, set
$$
I_\alpha(q) :=\pm q^{N} \frac{J_\alpha(q)^\vee-J_\alpha(q)}{1-q}
$$  
for some $N \in \ZZ$, where we take $\pm q^{N}$ so that $I_\alpha(q)\in \ZZ[q]$ with $I_\alpha(0)>0.$ Then we have that $I_\alpha(q)$ equals $\Tr(A)$ for some $A \in \mathrm{PSL}_q(2,\ZZ).$ In particular, $I_\alpha(0)=1$ and $I_\alpha(q)\in \ZZ_{\ge 0}[q].$ Moreover, we show that \Oguz\!'s conjecture \cite[Conjecture~1.4]{KO25} on circular fence posets implies a conjecture on $I_\alpha(q)$ (Conjecture~\ref{Oguz}). In particular, $I_\alpha(q)$ is conjectured to be at most bimodal 
(unimodal in most cases).


\section{The left and right \texorpdfstring{$q$}{q}-deformed rationals and  \texorpdfstring{${\rm PSL}_q(2,\ZZ)$}{PSL\_q(2,Z)}}\label{Pre}

\noindent
At the beginning, we briefly recall the definitions of the left and right $q$-deformed rational numbers. For details, see \cite{MO1, MOV, BBL, XR2}. We also give a few supplemental results. 

\subsection{Continued fractions, \texorpdfstring{${\rm PSL}(2,\ZZ)$}{PSL(2,Z)}-actions and its \textit{q}-deformation}

\noindent
In this paper, we regard $0= \frac{0}1$ and $\frac{1}{0}$ as irreducible fractions.  For an irreducible fraction $\frac{r}{s} \in \QQ $, we always assume that $s$ is positive.    
An irreducible fraction $\frac{r}{s} \in \QQ \cup \left\{ \frac{1}{0} \right\}$ has unique regular and negative continued fraction expansions as follows:\\
\[ \displaystyle
\frac{r}{s}=a_1+\frac{1}{a_2+\cfrac{1}{\ddots+\cfrac{1}{a_{2m}}}}
=c_1-\frac{1}{c_2-\cfrac{1}{\ddots-\cfrac{1}{c_k}}}
\]
\\
with the following setting. (The latter is often called the {\it Hirzebruch-Jung continued fraction}.)
If $\frac{r}{s}$ is positive, then we set
$a_1 \in \mathbb{Z}_{\geq 0}$, $a_i \in \mathbb{Z}_{\geq 1} $ ($i\geq2$), and $c_1 \in \mathbb{Z}_{\geq 1}$ and $c_j \in \mathbb{Z}_{\geq 2} $ ($j\geq2$). If $\frac{r}{s}$ is negative, then we set $a_1 \in \mathbb{Z}_{\leq 0}$, $a_i \in \mathbb{Z}_{\leq -1} $ ($i\geq2$), and $c_1 \in \mathbb{Z}_{\leq -1}$ and $c_j \in \mathbb{Z}_{\leq -2} $ ($j\geq2$). We write $[a_1,\ldots, a_{2m}]$ and $[[c_1,\ldots, c_k]]$ for these expansions, respectively. As exceptional cases, we set $\frac{0}{1}=[-1,1]=[[1,1]]$ and 
$\frac{1}{0}=[\; ]=[[\; ]]$ (i.e., empty expansions). For a regular continued fraction of odd length, we use the symbol $[a_1, \ldots, a_{2m+1}]$ also. Note that $[a_1,\ldots, a_{n}+1]=[a_1,\ldots, a_{n},1]$, and this is the reason we can assume that the length is even. 

Consider the group $\mathrm{SL}(2,\mathbb{Z})=\left\langle R,S \right\rangle=\left\langle R,L \right\rangle,$ where
$$
R:=\begin{pmatrix}
1 & 1 \\
0 & 1 \\
\end{pmatrix}, \qquad
S:=\begin{pmatrix}
0 & -1 \\
1 & 0 \\
\end{pmatrix}, \qquad
L:=\begin{pmatrix}
1 & 0 \\
1 & 1 \\
\end{pmatrix}.
$$
This group (and its quotient $\mathrm{PSL}(2,\mathbb{Z})$) acts on $\QQ \cup \left\{\frac{1}{0}\right\}$ by the fractional linear transformation:
\[
\displaystyle
\begin{pmatrix}
a & b \\
c & d \\
\end{pmatrix}\left(\frac{r}{s}\right)=\frac{ar+bs}{cr+ds},\]
where
$\displaystyle
\begin{pmatrix}
a & b \\
c & d \\
\end{pmatrix} \in \mathrm{SL}(2,\mathbb{Z})$ and $\frac{r}{s}\in \mathbb{Q}\cup \left\{\frac{1}{0} \right\}$. Hence, a rational number $\frac{r}{s}=[a_1,\ldots,a_{2m}]=[[c_1,\ldots, c_k]]$ can be expressed by the following formulas:
\begin{align}
\frac{r}{s}&=M(a_1,\ldots, a_{2m})\left(\frac{1}{0}\right):=R^{a_1}L^{a_2}R^{a_3}L^{a_4}\cdots R^{a_{2m-1}} L^{a_{2m}}\left(\frac{1}{0}\right);\label{R_r_over_s-1}\\
\frac{r}{s}&=M^{-}(c_1,\ldots, c_k)\left(\frac{1}{0}\right):=R^{c_1}SR^{c_2}S\cdots R^{c_{k}}S\left(\frac{1}{0}\right)\label{R_r_over_s-2}.
\end{align}

\begin{lem}\label{reduced form}
Let $\frac{r}s=[a_1, \ldots, a_{2m}]>0$ be an irreducible fraction with $s \ge 2$. For $\begin{pmatrix}
r & v \\
s & u \\
\end{pmatrix} \in \mathrm{SL}(2,\mathbb{Z})$, the following are equivalent.  
\begin{itemize}
\item[(1)]  $M(a_1,\ldots, a_{2m})=\begin{pmatrix}
r & v \\
s & u \\
\end{pmatrix}$.  
\item[(2)] $0 < u < s$.
\end{itemize}
\end{lem}

\begin{proof}
First, note that since $r, s>0$ and $ru-sv=1$, if the condition (2) holds then $r > v >0$.   
(1) $\Rightarrow$ (2): We can prove this by induction on $m$.  (2) $\Rightarrow$ (1): Any element in $\mathrm{SL}(2,\mathbb{Z})$ of the form $\begin{pmatrix}
r & * \\
s & * \\
\end{pmatrix}$ satisfies  $\begin{pmatrix}
r & v+nr \\
s & u+ns \\
\end{pmatrix}$ for some $n \in \ZZ$. There is only one $n$ satisfying $0 < u +ns <s$. 
\end{proof}

We first consider the $q$-deformation of \eqref{R_r_over_s-1}.

\begin{prop}[{\cite[Proposition 1.1]{LM}}]\label{def of q-rat}
Let $q$ be a formal parameter, and consider the matrices 
\[\displaystyle
R_q:=\begin{pmatrix}
q & 1 \\
0 & 1 \\
\end{pmatrix}, \ \ \ \ 
S_q=\begin{pmatrix}
0 & -q^{-1} \\
1 & 0 \\
\end{pmatrix}. \ \ \ \
\]
The subgroup $\<R_q, S_q \>$ of $\operatorname{GL}(2,\ZZ[q^{\pm 1}])$ generated by $R_q$ and $S_q$ contains $\pm q^{\pm 1} \mathrm{Id}$. We set  $$\PSL_q\left(2,\ZZ\right) :=\< R_q, S_q \>/\left\{\pm q^n \mathrm{Id}\mid n\in \mathbb{Z}\right\}.$$
Then, the assignment 
$$
R \longmapsto R_q, \qquad S \longmapsto S_q
$$
induces an isomorphism from $\mathrm{PSL}(2,\mathbb{Z})$ to $\mathrm{PSL}_{q}(2,\mathbb{Z})$.
\end{prop}

In the following, by abuse of notation, we sometimes identify $A \in \PSL_q\left(2,\ZZ\right)$  with an explicit matrix in $\mathrm{GL}(2,\ZZ[q^{\pm 1}])$ giving $A$.

We consider the matrix $$L_q:=q^{-1}R_qS_qR_q=\begin{pmatrix}
1 & 0 \\
1 & q^{-1} \\
\end{pmatrix},$$ then $R_q$ and $L_q$ also generate $\mathrm{PSL}_{q}(2,\mathbb{Z})$, and we set $$M_q(a_1,\ldots, a_{2m}):=R_q^{a_1}L_q^{a_2}R_q^{a_3}L_q^{a_4}\cdots R_q^{a_{2m-1}} L_q^{a_{2m}} \in \mathrm{GL}\left(2,\ZZ[q^{\pm 1}]\right).$$

\subsection{The left and right \textit{q}-deformed rational numbers}

\noindent
We recall the definition of the left and right \textit{q}-deformed rational numbers. We follow the notation in \cite{BBL}.
\begin{defn}[c.f. \cite{BBL}]

We set
\[
\displaystyle
 \ \ \ 
\left[\frac{0}{1}\right]^{\flat}_q=\frac{1-q^{-1}}{1}, \ \ \ 
\left[\frac{0}{1}\right]^{\sharp}_q=\frac{0}{1};\ \ \
\left[\frac{1}{0}\right]^{\flat}_q=\frac{1}{1-q}, \ \ \ 
\left[\frac{1}{0}\right]^{\sharp}_q=\frac{1}{0}. \ \ \ 
\]

For an irreducible fraction $\frac{r}{s}=[a_1,\ldots,a_{2m}] \in \mathbb{Q}$, 
then the {\it left $q$-deformed rational number} $\left[\displaystyle\frac{r}{s}\right]^\flat_q$ is defined by
\[\displaystyle
\left[\displaystyle\frac{r}{s}\right]^\flat_q:=M_q(a_1,\ldots, a_{2m})\left(\left[\frac{1}{0}\right]^{\flat}_q\right)=M_q(a_1,\ldots, a_{2m})\left(\frac{1}{1-q}\right),
\]
and the {\it right $q$-deformed rational number} 
$\left[\displaystyle\frac{r}{s}\right]^\sharp_q$ is defined by
\[\displaystyle
\left[\displaystyle\frac{r}{s}\right]^\sharp_q:=M_q(a_1,\ldots, a_{2m})\left(\left[\frac{1}{0}\right]^{\sharp}_q\right)=M_q(a_1,\ldots, a_{2m})\left(\frac{1}{0}\right).
\]
Multiplying both the numerator and denominator by 
$ \pm q^n$ for suitable $n \in \ZZ$, we have 
$$
\left[\displaystyle\frac{r}{s}\right]^\flat_q=\frac{\flR_{\frac{r}{s}}(q)}{\flS_{\frac{r}{s}}(q)}, \qquad 
\left[\displaystyle\frac{r}{s}\right]^\sharp_q=\frac{\shR_{\frac{r}{s}}(q)}{\shS_{\frac{r}{s}}(q)}, 
$$
where
$$  \flR_{\frac{r}{s}}(q), \ \shR_{\frac{r}{s}}(q)  \in \mathbb{Z}[q^{\pm 1}], \  \flS_{\frac{r}{s}}(q),  \shS_{\frac{r}{s}}(q)  \in \mathbb{Z}[q]$$
with $\flS_{\frac{r}{s}}(0)=\shS_{\frac{r}{s}}(0)=1$. 
\end{defn}

We always have 
\begin{itemize}
\item[(1)]$ \flR_{\frac{r}{s}}(1)=\shR_{\frac{r}{s}}(1) =r$ and $ \flS_{\frac{r}{s}}(1)=\shS_{\frac{r}{s}}(1) =s$.   
\item[(2)] $\flR_{\frac{r}{s}}(q)$ and $\flS_{\frac{r}{s}}(q)$ are coprime, and the same is true for 
$\shR_{\frac{r}{s}}(q)$ and $\shS_{\frac{r}{s}}(q)$. 
\end{itemize}
It is known that if $\frac{r}{s}>1$, then  
$ \flR_{\frac{r}{s}}(q), \ \shR_{\frac{r}{s}}(q) \in \ZZ[q]$ with  $ \flR_{\frac{r}{s}}(0)=\shR_{\frac{r}{s}}(0) =1.$\\ 

Easy calculation shows that $\flS_{\frac{n}1}(q)=\shS_{\frac{n}1}(q)=1$ for all $n \in \ZZ$, and 
$$
\flR_{\frac{n}1}(q) = \begin{cases} 
q^{n}+q^{n-2}+q^{n-3}+\cdots+q+1 & (n >0) \\
1-q^{-1} & (n=0)\\
-q^{-n-1}-q^{-n+1}-\cdots-q^{-1} & (n<0), 
\end{cases}
$$
while $\shR_{\frac{n}1}(q)$
coincides with the Euler $q$-integer $[n]_q=\dfrac{1-q^n}{1-q}$. 

We can use negative continued fractions to compute the $q$-deformed rational numbers.

\begin{prop}[c.f.  \cite{MO1,LM}]\label{negative version}
For the irreducible fraction $\frac{r}{s}=[[c_1,\ldots,c_k]]>0$, we consider the matrix 
$$M^{-}_q(c_1,\ldots, c_k):=R_q^{c_1}S_qR_q^{c_2}S_q\cdots R_q^{c_{k}}S_q \in \mathrm{GL}\left(2,\mathbb{Z}[q^{\pm 1}]\right),$$ then we have
$$
M^{-}_q(c_1,\ldots, c_k)=\begin{pmatrix}
\shR_{\frac{r}{s}}(q) & \mathcal{W}(q) \\
\shS_{\frac{r}{s}} (q) & \mathcal{X}(q) \\
\end{pmatrix}.
$$

Hence, the right $q$-deformed rational number of $\frac{r}{s}$ is 
\[\displaystyle
\left[\displaystyle\frac{r}{s}\right]^\sharp_q=M_q^{-}(c_1,\ldots, c_k)\left(\frac{1}{0}\right)=\frac{\shR_{\frac{r}{s}}(q)}{\shS_{\frac{r}{s}}(q)}.
\]
\end{prop}
\noindent
The first author of the present paper gave the left version of the above proposition (c.f. \cite[Section~2.4]{XR2}).

\medskip

Proposition~\ref{negative version} can be extended as follows.  
The multiplicative group $G:=\{ \pm q^n \mid n \in \ZZ \}$ acts on the set 
$$
X:=
\left\{ \begin{pmatrix}
\mathcal{V} \\
 \mathcal{U} \\
\end{pmatrix} \mid \mathcal{V}, \mathcal{U} \in \ZZ[q^{\pm 1}]
\right\}
$$
in the natural way. 
Let $\overline{X}:= X/G$, and  
$\left[ \dfrac{\cV}{\cU} \right] \in \overline{X}$ denote the orbit of $\begin{pmatrix}
\mathcal{V} \\
 \mathcal{U} \\
\end{pmatrix}$.  
Clearly, $\mathrm{PSL}_q(2,\mathbb{Z})$ acts on $\overline{X}$.  
Consider  the subset 
$$
\overline{Y}:=\left\{ \left[ \frac{\shR_{\frac{r}{s}}(q)}{\shS_{\frac{r}{s}}(q)} \right] \ \middle | \ \frac{r}s \in \QQ \cup \left\{\frac{1}{0}\right\}\right\}
$$
of $\overline{X}$. 
The following seems to be well-known to experts.

\begin{prop}\label{column of a general element} 
With the above notation, the columns of any element of $\mathrm{PSL}_{q}(2,\mathbb{Z})$ belong to $\overline{Y}$. Hence, for 
$$\begin{pmatrix}
\cR(q) & \mathcal{V}(q) \\
\cS(q) &  \mathcal{U}(q) \\
\end{pmatrix} \in \mathrm{PSL}_{q}(2,\mathbb{Z})
$$   
with $r:=\cR(1)$, $s:=\cS(1)$, $v:=\mathcal{V}(1)$ and 
$u:=\mathcal{U}(1)$, we have 
$$
\left[\displaystyle\frac{r}{s}\right]^{\sharp}_q= \left[\frac{\cR(q)}{\cS(q)} \right] \quad \text{and} \quad  \left[\displaystyle\frac{v}{u}\right]^{\sharp}_q= \left[\frac{\mathcal{V}(q)}{\mathcal{U}(q)}\right]
$$
(strictly speaking, if $s <0$, $\frac{r}s$ should be $\frac{-r}{-s}$, and so on). 
\end{prop}

\begin{proof}
Since $R_q$ and $S_q$ generate $\mathrm{PSL}_q(2,\mathbb{Z})$, it suffices to show that $\overline{Y}$ is closed under multiplication by these matrices. 
An easy computation shows that  
$$R_q \left[\frac{\shR_{\frac{r}s}(q)}{\shS_{\frac{r}s}(q)} \right]=\begin{pmatrix}
q & 1 \\
0 & 1 \\
\end{pmatrix}\left[\frac{\shR_{\frac{r}s}(q)}{\shS_{\frac{r}s}(q)} \right]=
\left[\frac{q\shR_{\frac{r}s}(q)+\shS_{\frac{r}s}(q)}{\shS_{\frac{r}s}(q)} \right]
=\left[\frac{\shR_{\frac{r}s+1}(q)}{\shS_{\frac{r}s+1}(q)} \right] \in \overline{Y}, 
$$
where the last equality follows from \cite[Proposition 2.8 (2.6)]{LM}.  
Similarly, we have 
$$S_q \left[\frac{\shR_{\frac{r}s}(q)}{\shS_{\frac{r}s}(q)} \right]=\begin{pmatrix}
0 & -q^{-1} \\
1 & 0 \\
\end{pmatrix}\left[\frac{\shR_{\frac{r}s}(q)}{\shS_{\frac{r}s}(q)} \right]=
\left[\frac{-q^{-1}\shS_{\frac{r}s}(q)}{\shR_{\frac{r}s}(q)} \right]
=\left[\frac{\shR_{-\frac{s}r}(q)}{\shS_{-\frac{s}r}(q)} \right] \in \overline{Y}, 
$$
where the last equality follows from \cite[Proposition 2.8 (2.8)]{LM}.     
\end{proof}

\begin{conven}\label{equiv}
{\normalfont
For $f,g \in \ZZ[q^{\pm 1}]$, if $f=\pm q^n g$ for some $n \in \ZZ$, then we write $f \equiv g$. And for matrices $A,B$, if $A=\pm q^n B$ for some $n \in \ZZ$, then we write $A \equiv B$. }
\end{conven}

\subsection{Closures of a quiver for the  left and right \texorpdfstring{$q$}{q}-deformed rational numbers} 
We recall a combinatorial interpretation of the left and right $q$-deformed rational numbers. For details, see \cite{MO1,BBL,KMRWY, TBEC1}.

A \textit{quiver} $Q=(Q_0,Q_1,s,t)$ consists of two sets $Q_0$, $Q_1$ and two maps $s,t:Q_1\to Q_0$. Each element of $Q_0$ (resp. $Q_1$) is called a vertex (resp. an arrow). For an arrow $\alpha\in Q_1$, we call $s(\alpha)$ (resp. $t(\alpha)$) the source (resp. the target) of $\alpha$. 
We will commonly write 
$\alpha:a\to b$ to indicate that an arrow $\alpha$ has the source $a$ and the target $b$. 
In this paper, we always consider a  \textit{finite} quiver $Q$ (i.e., both two sets $Q_0$ and $Q_1$ are finite sets).
The \textit{opposite quiver} of $Q$, say $Q^\vee$, is defined by $Q^\vee=(Q_0,Q_1,t,s)$.

Let $Q$ be a finite quiver.
A subset $C\subset Q_0$ is a \textit{closure} if there is no arrow $\alpha\in Q_1$ such that $s(\alpha)\in C$ and $t(\alpha)\in Q_0\setminus C$.
A closure $C$ is an \textit{$\ell$-closure} if the number of elements of $C$ is $\ell$.
The number of $\ell$-closures is denoted by $\rho_\ell(Q)$.
Then the polynomial 
\[ \mathsf{cl}(Q;q):=\sum_{\ell=0}^{n}\rho_\ell(Q)q^{\ell}\in\mathbb{Z}[q],\]
where $n=|Q_0|$, is called the \textit{closure polynomial} of $Q$.

Clearly, $C\subset Q_0$ is a closure of $Q$ if and only if $Q_0\setminus C$ is that of $Q^\vee.$ Hence, one has
\begin{equation}\label{op-cl}
\mathsf{cl}\left(Q^\vee;q\right)=\mathsf{cl}\left(Q;q\right)^\vee.
\end{equation}

For a tuple of integers $\mathbf{b}:=(b_1,b_2,\ldots, b_d)$ with $b_1, b_d\geq 0$, $b_2,\ldots,b_{d-1} >0$, 
we set the quiver $Q(\mathbf{b})$ of type $A$ as follows (in this figure,  we assume that $d$ is even):

\begin{figure}[ht]
\begin{center}
\begin{tikzpicture}[x=0.58pt,y=0.6pt,yscale=-1,xscale=0.85]

\draw    (75.8,45.16) -- (36.11,45.41) ;
\draw [shift={(34.11,45.42)}, rotate = 359.64] [color={rgb, 255:red, 0; green, 0; blue, 0 }  ][line width=0.75]    (10.93,-3.29) .. controls (6.95,-1.4) and (3.31,-0.3) .. (0,0) .. controls (3.31,0.3) and (6.95,1.4) .. (10.93,3.29)   ;

\draw    (204.18,45.13) -- (243.8,44.77) ;
\draw [shift={(245.8,44.76)}, rotate = 179.49] [color={rgb, 255:red, 0; green, 0; blue, 0 }  ][line width=0.75]    (10.93,-3.29) .. controls (6.95,-1.4) and (3.31,-0.3) .. (0,0) .. controls (3.31,0.3) and (6.95,1.4) .. (10.93,3.29)   ;

\draw    (186.6,45.16) -- (146.91,45.41) ;
\draw [shift={(144.91,45.42)}, rotate = 359.64] [color={rgb, 255:red, 0; green, 0; blue, 0 }  ][line width=0.75]    (10.93,-3.29) .. controls (6.95,-1.4) and (3.31,-0.3) .. (0,0) .. controls (3.31,0.3) and (6.95,1.4) .. (10.93,3.29)   ;

\draw   (24.8,57) .. controls (24.79,61.67) and (27.12,64) .. (31.79,64.01) -- (99.09,64.06) .. controls (105.76,64.07) and (109.09,66.4) .. (109.09,71.07) .. controls (109.09,66.4) and (112.42,64.07) .. (119.09,64.08)(116.09,64.08) -- (186.39,64.14) .. controls (191.06,64.15) and (193.39,61.82) .. (193.4,57.15) ;

\draw   (197,57.56) .. controls (197.01,62.23) and (199.34,64.56) .. (204.01,64.55) -- (270.51,64.49) .. controls (277.18,64.48) and (280.51,66.81) .. (280.51,71.48) .. controls (280.51,66.81) and (283.84,64.48) .. (290.51,64.47)(287.51,64.48) -- (357.01,64.42) .. controls (361.68,64.41) and (364.01,62.08) .. (364,57.41) ;

\draw    (314.18,44.73) -- (353.8,44.37) ;
\draw [shift={(355.8,44.36)}, rotate = 179.49] [color={rgb, 255:red, 0; green, 0; blue, 0 }  ][line width=0.75]    (10.93,-3.29) .. controls (6.95,-1.4) and (3.31,-0.3) .. (0,0) .. controls (3.31,0.3) and (6.95,1.4) .. (10.93,3.29)   ;

\draw    (417.68,45.16) -- (377.99,45.41) ;
\draw [shift={(375.99,45.42)}, rotate = 359.64] [color={rgb, 255:red, 0; green, 0; blue, 0 }  ][line width=0.75]    (10.93,-3.29) .. controls (6.95,-1.4) and (3.31,-0.3) .. (0,0) .. controls (3.31,0.3) and (6.95,1.4) .. (10.93,3.29)   ;

\draw    (528.48,45.16) -- (488.79,45.41) ;
\draw [shift={(486.79,45.42)}, rotate = 359.64] [color={rgb, 255:red, 0; green, 0; blue, 0 }  ][line width=0.75]    (10.93,-3.29) .. controls (6.95,-1.4) and (3.31,-0.3) .. (0,0) .. controls (3.31,0.3) and (6.95,1.4) .. (10.93,3.29)   ;

\draw   (366.68,57) .. controls (366.67,61.67) and (369,64) .. (373.67,64.01) -- (440.97,64.06) .. controls (447.64,64.07) and (450.97,66.4) .. (450.97,71.07) .. controls (450.97,66.4) and (454.3,64.07) .. (460.97,64.08)(457.97,64.08) -- (528.27,64.14) .. controls (532.94,64.15) and (535.27,61.82) .. (535.28,57.15) ;

\draw    (594.18,45.63) -- (633.8,45.27) ;
\draw [shift={(635.8,45.26)}, rotate = 179.49] [color={rgb, 255:red, 0; green, 0; blue, 0 }  ][line width=0.75]    (10.93,-3.29) .. controls (6.95,-1.4) and (3.31,-0.3) .. (0,0) .. controls (3.31,0.3) and (6.95,1.4) .. (10.93,3.29)   ;

\draw   (587,58.06) .. controls (587.01,62.73) and (589.34,65.06) .. (594.01,65.05) -- (660.51,64.99) .. controls (667.18,64.98) and (670.51,67.31) .. (670.51,71.98) .. controls (670.51,67.31) and (673.84,64.98) .. (680.51,64.97)(677.51,64.98) -- (747.01,64.92) .. controls (751.68,64.91) and (754.01,62.58) .. (754,57.91) ;

\draw    (704.18,45.23) -- (743.8,44.87) ;
\draw [shift={(745.8,44.86)}, rotate = 179.49] [color={rgb, 255:red, 0; green, 0; blue, 0 }  ][line width=0.75]    (10.93,-3.29) .. controls (6.95,-1.4) and (3.31,-0.3) .. (0,0) .. controls (3.31,0.3) and (6.95,1.4) .. (10.93,3.29)   ;

\draw (15.12,40) node [anchor=north west][inner sep=0.75pt]  [font=\Large]  {$\circ $};

\draw (744.97,40) node [anchor=north west][inner sep=0.75pt]  [font=\Large]  {$\circ $};

\draw (75.6,40) node [anchor=north west][inner sep=0.75pt]  [font=\Large]  {$\circ $};

\draw (92.8,40) node [anchor=north west][inner sep=0.75pt]  [font=\large]  {$\cdots $};

\draw (245.48,40) node [anchor=north west][inner sep=0.75pt]  [font=\Large]  {$\circ $};

\draw (184.8,40) node [anchor=north west][inner sep=0.75pt]  [font=\Large]  {$\circ $};

\draw (124.8,40) node [anchor=north west][inner sep=0.75pt]  [font=\Large]  {$\circ $};

\draw (67.2,76.2) node [anchor=north west][inner sep=0.75pt]    {$b_{1} \ \text{left arrows}$};

\draw (236,76.6) node [anchor=north west][inner sep=0.75pt]    {$b_{2} \ \text{right arrows}$};

\draw (263.2,40) node [anchor=north west][inner sep=0.75pt]  [font=\large]  {$\cdots $};

\draw (295.2,40) node [anchor=north west][inner sep=0.75pt]  [font=\Large]  {$\circ $};

\draw (355.48,40) node [anchor=north west][inner sep=0.75pt]  [font=\Large]  {$\circ $};

\draw (417.48,40) node [anchor=north west][inner sep=0.75pt]  [font=\Large]  {$\circ $};

\draw (434.68,40) node [anchor=north west][inner sep=0.75pt]  [font=\large]  {$\cdots $};

\draw (466.68,40) node [anchor=north west][inner sep=0.75pt]  [font=\Large]  {$\circ $};

\draw (409.08,77.1) node [anchor=north west][inner sep=0.75pt]    {$b_{3} \ \text{left arrows}$};

\draw (525.98,40) node [anchor=north west][inner sep=0.75pt]  [font=\Large]  {$\circ $};

\draw (542.7,40) node [anchor=north west][inner sep=0.75pt]  [font=\large]  {$\cdots $};

\draw (635.48,40) node [anchor=north west][inner sep=0.75pt]  [font=\Large]  {$\circ $};

\draw (574.8,40) node [anchor=north west][inner sep=0.75pt]  [font=\Large]  {$\circ $};

\draw (626,77.1) node [anchor=north west][inner sep=0.75pt]    {$b_{d} \ \text{right arrows}$};

\draw (653.2,40) node [anchor=north west][inner sep=0.75pt]  [font=\large]  {$\cdots $};

\draw (685.2,40) node [anchor=north west][inner sep=0.75pt]  [font=\Large]  {$\circ $};

\end{tikzpicture}

\end{center}
\caption{The quiver $Q(\bb).$}
   \label{Q-sharp} 
\end{figure}

For a rational number $\frac{r}{s}=[a_1, a_2, \ldots, a_{2m}]>1$, we set 

\begin{align*}
& Q_{\frac{r}{s}}^{\sharp,\cR} :=Q(a_1-1, a_2, \ldots , a_{2m-1}, a_{2m}-1), \\
& Q_{\alpha}^{\sharp, \cS} :=\left\{\begin{array}{ll}
Q(0,a_2-1, a_3, \ldots , a_{2m-1}, a_{2m}-1)  & \text{if $m>1$,} \\ [5pt]
Q(0,a_2-2)  & \text{if $m=1$.}
\end{array}\right.
\end{align*}
Here, if $a_2=1$ and $m>1$ (resp. $a_2=2$ and $m=1$, $a_2=1$ and $m=1$), we understand that $Q_{\frac{r}{s}}^{\sharp,\mathcal{S}}=Q(a_3, \ldots , a_{2m-1}, a_{2m}-1)$ (resp. $Q_{\frac{r}{s}}^{\sharp,\mathcal{S}}={Q_{\frac{r}{2}}^{\sharp,\mathcal{S}}=}Q(0)$, $Q_{\frac{r}{s}}^{{\sharp}, \mathcal{S}}={Q_{\frac{r}{1}}^{\sharp,\mathcal{S}}=}\varnothing$).
The quiver $Q_{\frac{r}{s}}^{\sharp,\cS}$ is obtained by deleting the first $a_1$ arrows from $Q_{\frac{r}{s}}^{\sharp,\cR}$.

For the right $q$-deformed rational numbers, we have the following result.

\begin{thm}[{\cite[Theorem 4]{MO1}}]\label{theorem:MO closure} 
Let $\frac{r}{s} >1$ be an irreducible fraction.
Then, the following equations hold:
\begin{align}
\mathcal{R}^{\sharp}_{\frac{r}{s}}(q) &=\mathsf{cl}\left(Q_{\frac{r}{s}}^{\sharp,\cR};q\right), \\
\mathcal{S}^{\sharp}_{\frac{r}{s}}(q) &=\mathsf{cl}\left(Q_{\frac{r}{s}}^{\sharp,\cS};q\right).
\end{align}
\end{thm} 

Let $P$ be a finite poset. 
A subset $I$ of $P$ is said to be a {\it lower order ideal}, if  $x  \in I$, $y \in P$ and $y \le x$ imply $y \in I$.  Let $J(P)$ be the set of lower order ideals of $P$. The polynomial 
$$\mathsf{rk}(P; q):= \sum_{I \in J(P)}q ^{|I|}$$
is called the {\it rank polynomial} of $P$.

We can regard (the vertex set $Q_0$ of) the quiver  $Q(\bb)$ as a poset as follows; for $a, b \in Q_0$, $a < b$ if and only if there is an arrow $b \to a$. Clearly, $I \subset Q(\bb)$ is a closure if and only if it is a lower order ideal, and hence 
$$\mathsf{rk}(Q(\bb); q)=\mathsf{cl}(Q(\bb);q).$$

For the left $q$-deformed rational numbers, we also have a corresponding quiver. 

For a tuple of integers $\bb:=(b_1,b_2,\ldots, b_d)$ with $b_1, b_d\geq 0$, $b_2,\ldots,b_{d-1} >0$, 
we define the quiver $Q^\flat(\bb)$ adding a cycle to the right side of the quiver $Q(\bb)$ as follows (here we assume that $d$ is even again).
\begin{figure}[ht]
\begin{tikzpicture}[x=0.525pt,y=0.65pt,yscale=-1,xscale=0.85]

\draw    (764.86,40.26) .. controls (785.7,-2.33) and (806.41,15.57) .. (818.31,38.2) ;
\draw [shift={(819.21,39.95)}, rotate = 243.39] [color={rgb, 255:red, 0; green, 0; blue, 0 }  ][line width=0.75]    (10.93,-3.29) .. controls (6.95,-1.4) and (3.31,-0.3) .. (0,0) .. controls (3.31,0.3) and (6.95,1.4) .. (10.93,3.29)   ;

\draw    (817.94,53.13) .. controls (795.9,79.96) and (778.25,75.91) .. (767.22,53.34) ;
\draw [shift={(766.38,51.57)}, rotate = 65.55] [color={rgb, 255:red, 0; green, 0; blue, 0 }  ][line width=0.75]    (10.93,-3.29) .. controls (6.95,-1.4) and (3.31,-0.3) .. (0,0) .. controls (3.31,0.3) and (6.95,1.4) .. (10.93,3.29)   ;

\draw    (75.8,45.16) -- (36.11,45.41) ;
\draw [shift={(34.11,45.42)}, rotate = 359.64] [color={rgb, 255:red, 0; green, 0; blue, 0 }  ][line width=0.75]    (10.93,-3.29) .. controls (6.95,-1.4) and (3.31,-0.3) .. (0,0) .. controls (3.31,0.3) and (6.95,1.4) .. (10.93,3.29)   ;

\draw    (204.18,45.13) -- (243.8,44.77) ;
\draw [shift={(245.8,44.76)}, rotate = 179.49] [color={rgb, 255:red, 0; green, 0; blue, 0 }  ][line width=0.75]    (10.93,-3.29) .. controls (6.95,-1.4) and (3.31,-0.3) .. (0,0) .. controls (3.31,0.3) and (6.95,1.4) .. (10.93,3.29)   ;

\draw    (186.6,45.16) -- (146.91,45.41) ;
\draw [shift={(144.91,45.42)}, rotate = 359.64] [color={rgb, 255:red, 0; green, 0; blue, 0 }  ][line width=0.75]    (10.93,-3.29) .. controls (6.95,-1.4) and (3.31,-0.3) .. (0,0) .. controls (3.31,0.3) and (6.95,1.4) .. (10.93,3.29)   ;

\draw   (24.8,57) .. controls (24.79,61.67) and (27.12,64) .. (31.79,64.01) -- (99.09,64.06) .. controls (105.76,64.07) and (109.09,66.4) .. (109.09,71.07) .. controls (109.09,66.4) and (112.42,64.07) .. (119.09,64.08)(116.09,64.08) -- (186.39,64.14) .. controls (191.06,64.15) and (193.39,61.82) .. (193.4,57.15) ;

\draw   (197,57.56) .. controls (197.01,62.23) and (199.34,64.56) .. (204.01,64.55) -- (270.51,64.49) .. controls (277.18,64.48) and (280.51,66.81) .. (280.51,71.48) .. controls (280.51,66.81) and (283.84,64.48) .. (290.51,64.47)(287.51,64.48) -- (357.01,64.42) .. controls (361.68,64.41) and (364.01,62.08) .. (364,57.41) ;

\draw    (314.18,44.73) -- (353.8,44.37) ;
\draw [shift={(355.8,44.36)}, rotate = 179.49] [color={rgb, 255:red, 0; green, 0; blue, 0 }  ][line width=0.75]    (10.93,-3.29) .. controls (6.95,-1.4) and (3.31,-0.3) .. (0,0) .. controls (3.31,0.3) and (6.95,1.4) .. (10.93,3.29)   ;

\draw    (417.68,45.16) -- (377.99,45.41) ;
\draw [shift={(375.99,45.42)}, rotate = 359.64] [color={rgb, 255:red, 0; green, 0; blue, 0 }  ][line width=0.75]    (10.93,-3.29) .. controls (6.95,-1.4) and (3.31,-0.3) .. (0,0) .. controls (3.31,0.3) and (6.95,1.4) .. (10.93,3.29)   ;

\draw    (528.48,45.16) -- (488.79,45.41) ;
\draw [shift={(486.79,45.42)}, rotate = 359.64] [color={rgb, 255:red, 0; green, 0; blue, 0 }  ][line width=0.75]    (10.93,-3.29) .. controls (6.95,-1.4) and (3.31,-0.3) .. (0,0) .. controls (3.31,0.3) and (6.95,1.4) .. (10.93,3.29)   ;

\draw   (366.68,57) .. controls (366.67,61.67) and (369,64) .. (373.67,64.01) -- (440.97,64.06) .. controls (447.64,64.07) and (450.97,66.4) .. (450.97,71.07) .. controls (450.97,66.4) and (454.3,64.07) .. (460.97,64.08)(457.97,64.08) -- (528.27,64.14) .. controls (532.94,64.15) and (535.27,61.82) .. (535.28,57.15) ;

\draw    (594.18,45.63) -- (633.8,45.27) ;
\draw [shift={(635.8,45.26)}, rotate = 179.49] [color={rgb, 255:red, 0; green, 0; blue, 0 }  ][line width=0.75]    (10.93,-3.29) .. controls (6.95,-1.4) and (3.31,-0.3) .. (0,0) .. controls (3.31,0.3) and (6.95,1.4) .. (10.93,3.29)   ;

\draw   (587,58.06) .. controls (587.01,62.73) and (589.34,65.06) .. (594.01,65.05) -- (660.51,64.99) .. controls (667.18,64.98) and (670.51,67.31) .. (670.51,71.98) .. controls (670.51,67.31) and (673.84,64.98) .. (680.51,64.97)(677.51,64.98) -- (747.01,64.92) .. controls (751.68,64.91) and (754.01,62.58) .. (754,57.91) ;

\draw    (704.18,45.23) -- (743.8,44.87) ;
\draw [shift={(745.8,44.86)}, rotate = 179.49] [color={rgb, 255:red, 0; green, 0; blue, 0 }  ][line width=0.75]    (10.93,-3.29) .. controls (6.95,-1.4) and (3.31,-0.3) .. (0,0) .. controls (3.31,0.3) and (6.95,1.4) .. (10.93,3.29)   ;

\draw (15.12,40) node [anchor=north west][inner sep=0.60pt]  [font=\large]  {$\circ $};
\draw (744.97,40) node [anchor=north west][inner sep=0.60pt]  [font=\large]  {$\circ $};
\draw (820.26,40) node [anchor=north west][inner sep=0.60pt]  [font=\large]  {$\circ $};
\draw (75.6,40) node [anchor=north west][inner sep=0.60pt]  [font=\large]  {$\circ $};
\draw (92.8,40) node [anchor=north west][inner sep=0.60pt]  [font=\large]  {$\cdots $};
\draw (245.48,40) node [anchor=north west][inner sep=0.60pt]  [font=\large]  {$\circ $};
\draw (184.8,40) node [anchor=north west][inner sep=0.60pt]  [font=\large]  {$\circ $};
\draw (124.8,40) node [anchor=north west][inner sep=0.60pt]  [font=\large]  {$\circ $};
\draw (67.2,76.2) node [anchor=north west][inner sep=0.60pt]    {$b_{1} \ \text{left arrows}$};
\draw (236,76.6) node [anchor=north west][inner sep=0.60pt]    {$b_{2} \ \text{right arrows}$};
\draw (263.2,40) node [anchor=north west][inner sep=0.60pt]  [font=\large]  {$\cdots $};
\draw (295.2,40) node [anchor=north west][inner sep=0.60pt]  [font=\large]  {$\circ $};
\draw (355.48,40) node [anchor=north west][inner sep=0.60pt]  [font=\large]  {$\circ $};
\draw (417.48,40) node [anchor=north west][inner sep=0.60pt]  [font=\large]  {$\circ $};
\draw (434.68,40) node [anchor=north west][inner sep=0.60pt]  [font=\large]  {$\cdots $};
\draw (466.68,40) node [anchor=north west][inner sep=0.60pt]  [font=\large]  {$\circ $};
\draw (409.08,77.1) node [anchor=north west][inner sep=0.60pt]    {$b_{3} \ \text{left arrows}$};
\draw (525.98,40) node [anchor=north west][inner sep=0.60pt]  [font=\large]  {$\circ $};
\draw (542.7,40) node [anchor=north west][inner sep=0.60pt]  [font=\large]  {$\cdots $};
\draw (635.48,40) node [anchor=north west][inner sep=0.60pt]  [font=\large]  {$\circ $};
\draw (574.8,40) node [anchor=north west][inner sep=0.60pt]  [font=\large]  {$\circ $};
\draw (626,77.1) node [anchor=north west][inner sep=0.60pt]    {$b_{d} \ \text{right arrows}$};
\draw (653.2,40) node [anchor=north west][inner sep=0.60pt]  [font=\large]  {$\cdots $};
\draw (685.2,40) node [anchor=north west][inner sep=0.60pt]  [font=\large]  {$\circ $};
\end{tikzpicture}
\caption{The quiver $Q^\flat(\bb).$}
   \label{Q-flat} 
\end{figure}
For a rational number $\frac{r}{s}=[a_1, a_2, \ldots, a_{2m}]>1$, set 
 $$Q_{\frac{r}{s}}^{\flat,\cR}:=Q^\flat(a_1-1, a_2, \ldots , a_{2m-1}, a_{2m}-1).$$
As in the right case, $Q_{\frac{r}{s}}^{\flat,\cS}$ is obtained by deleting the first $a_1$ arrows from $Q_{\frac{r}{s}}^{\flat,\cR}$.

Then for the left $q$-deformed rational numbers, we have the following result.

\begin{thm}[{\cite[Corollary A.2]{BBL}}]\label{theorem:BBL closure} 
Let $\frac{r}{s} >1$ be an irreducible fraction.
Then, the following equations hold:
\begin{align}
\mathcal{R}^{\flat}_{\frac{r}{s}}(q) &=\mathsf{cl}\left(Q_{\frac{r}{s}}^{\flat,\cR};q\right), \\
\mathcal{S}^{\flat}_{\frac{r}{s}}(q) &=\mathsf{cl}\left(Q_{\frac{r}{s}}^{\flat,\cS};q\right).
\end{align}
\end{thm}

\begin{exam}
    For $\frac{11}{8}=[1,2,1,2]$, we can draw the quiver $Q^{\flat,\cR}_{\frac{11}{8}}$ as in Figure~\ref{Q-flat-R-11/8},
\begin{figure}[ht]
\begin{tikzpicture}[x=0.5pt,y=0.5pt,yscale=-1,xscale=1]

\draw    (233.5,46.16) -- (193.81,46.41) ;
\draw [shift={(191.81,46.42)}, rotate = 359.64] [color={rgb, 255:red, 0; green, 0; blue, 0 }  ][line width=0.75]    (10.93,-3.29) .. controls (6.95,-1.4) and (3.31,-0.3) .. (0,0) .. controls (3.31,0.3) and (6.95,1.4) .. (10.93,3.29)   ;

\draw    (344.33,40.09) .. controls (365.18,-2.5) and (385.89,15.4) .. (397.79,38.03) ;
\draw [shift={(398.68,39.78)}, rotate = 243.39] [color={rgb, 255:red, 0; green, 0; blue, 0 }  ][line width=0.75]    (10.93,-3.29) .. controls (6.95,-1.4) and (3.31,-0.3) .. (0,0) .. controls (3.31,0.3) and (6.95,1.4) .. (10.93,3.29)   ;

\draw    (397.41,52.96) .. controls (375.38,79.79) and (357.72,75.74) .. (346.69,53.17) ;
\draw [shift={(345.86,51.4)}, rotate = 65.55] [color={rgb, 255:red, 0; green, 0; blue, 0 }  ][line width=0.75]    (10.93,-3.29) .. controls (6.95,-1.4) and (3.31,-0.3) .. (0,0) .. controls (3.31,0.3) and (6.95,1.4) .. (10.93,3.29)   ;

\draw    (33,45.66) -- (74,45.66) ;
\draw [shift={(76,45.66)}, rotate = 180] [color={rgb, 255:red, 0; green, 0; blue, 0 }  ][line width=0.75]    (10.93,-3.29) .. controls (6.95,-1.4) and (3.31,-0.3) .. (0,0) .. controls (3.31,0.3) and (6.95,1.4) .. (10.93,3.29)   ;

\draw    (109,46.66) -- (150,46.66) ;
\draw [shift={(152,46.66)}, rotate = 180] [color={rgb, 255:red, 0; green, 0; blue, 0 }  ][line width=0.75]    (10.93,-3.29) .. controls (6.95,-1.4) and (3.31,-0.3) .. (0,0) .. controls (3.31,0.3) and (6.95,1.4) .. (10.93,3.29)   ;

\draw    (272,46.66) -- (313,46.66) ;
\draw [shift={(315,46.66)}, rotate = 180] [color={rgb, 255:red, 0; green, 0; blue, 0 }  ][line width=0.75]    (10.93,-3.29) .. controls (6.95,-1.4) and (3.31,-0.3) .. (0,0) .. controls (3.31,0.3) and (6.95,1.4) .. (10.93,3.29)   ;

\draw (161.53,40) node [anchor=north west][inner sep=0.75pt]  [font=\Large]  {$\circ $};
\draw (11.4,40) node [anchor=north west][inner sep=0.75pt]  [font=\Large]  {$\circ $};
\draw (321.78,40) node [anchor=north west][inner sep=0.75pt]  [font=\Large]  {$\circ $};
\draw (399.73,40) node [anchor=north west][inner sep=0.75pt]  [font=\Large]  {$\circ $};
\draw (90,40) node [anchor=north west][inner sep=0.75pt]  [font=\Large]  {$\circ $};
\draw (242.23,40) node [anchor=north west][inner sep=0.75pt]  [font=\Large]  {$\circ $};
\end{tikzpicture}
\caption{The quiver $Q^{\flat,\cR}_{\frac{11}{8}}.$}
   \label{Q-flat-R-11/8} 
\end{figure}
and we have
$\mathsf{cl}\left(Q_{\frac{11}{8}}^{\flat,\cR};q\right) = q^{6}+2 q^{5}+2 q^{4}+2 q^{3}+2 q^{2}+q +1=\mathcal{R}^{\flat}_{\frac{11}{8}}(q).$

Similarly, 
\begin{figure}[ht]
\begin{center}
\begin{tikzpicture}[x=0.55pt,y=0.45pt,yscale=-1,xscale=1]

\draw    (277.93,40.09) .. controls (298.78,-2.5) and (319.49,15.4) .. (331.39,38.03) ;
\draw [shift={(332.28,39.78)}, rotate = 243.39] [color={rgb, 255:red, 0; green, 0; blue, 0 }  ][line width=0.75]    (10.93,-3.29) .. controls (6.95,-1.4) and (3.31,-0.3) .. (0,0) .. controls (3.31,0.3) and (6.95,1.4) .. (10.93,3.29)   ;

\draw    (331.01,52.96) .. controls (308.98,79.79) and (291.32,75.74) .. (280.29,53.17) ;
\draw [shift={(279.46,51.4)}, rotate = 65.55] [color={rgb, 255:red, 0; green, 0; blue, 0 }  ][line width=0.75]    (10.93,-3.29) .. controls (6.95,-1.4) and (3.31,-0.3) .. (0,0) .. controls (3.31,0.3) and (6.95,1.4) .. (10.93,3.29)   ;

\draw    (43.18,44.79) -- (82.8,44.44) ;
\draw [shift={(84.8,44.42)}, rotate = 179.49] [color={rgb, 255:red, 0; green, 0; blue, 0 }  ][line width=0.75]    (10.93,-3.29) .. controls (6.95,-1.4) and (3.31,-0.3) .. (0,0) .. controls (3.31,0.3) and (6.95,1.4) .. (10.93,3.29)   ;

\draw    (166.47,44.85) -- (126.78,45.1) ;
\draw [shift={(124.78,45.11)}, rotate = 359.64] [color={rgb, 255:red, 0; green, 0; blue, 0 }  ][line width=0.75]    (10.93,-3.29) .. controls (6.95,-1.4) and (3.31,-0.3) .. (0,0) .. controls (3.31,0.3) and (6.95,1.4) .. (10.93,3.29)   ;

\draw    (206.18,44.79) -- (245.8,44.44) ;
\draw [shift={(247.8,44.42)}, rotate = 179.49] [color={rgb, 255:red, 0; green, 0; blue, 0 }  ][line width=0.75]    (10.93,-3.29) .. controls (6.95,-1.4) and (3.31,-0.3) .. (0,0) .. controls (3.31,0.3) and (6.95,1.4) .. (10.93,3.29)   ;

\draw (15.13,40) node [anchor=north west][inner sep=0.75pt]  [font=\Large]  {$\circ $};
\draw (255.38,40) node [anchor=north west][inner sep=0.75pt]  [font=\Large]  {$\circ $};
\draw (333.33,40) node [anchor=north west][inner sep=0.75pt]  [font=\Large]  {$\circ $};
\draw (175,40) node [anchor=north west][inner sep=0.75pt]  [font=\Large]  {$\circ $};
\draw (95.83,40) node [anchor=north west][inner sep=0.75pt]  [font=\Large]  {$\circ $};
\end{tikzpicture}
\end{center}
\caption{The quiver $Q^{\flat,\cS}_{\frac{11}{8}}.$}
   \label{Q-flat-S-11/8} 
\end{figure}
we have
$\mathsf{cl}\left(Q_{\frac{11}{8}}^{\flat,\cS};q\right) = q^{5}+2 q^{4}+q^{3}+2 q^{2}+q +1=\mathcal{S}^{\flat}_{\frac{11}{8}}(q).$
\end{exam}

\begin{rem}\label{odd length}
 If $\frac{r}{s} \not \in \ZZ$ and $\frac{r}{s} >1$, the above constructions of $Q_{\frac{r}{s}}^{\flat,\cR}$, $Q_{\frac{r}{s}}^{\flat,\cS}$ and $Q_{\frac{r}{s}}^{\sharp,\cS}$, $Q_{\frac{r}{s}}^{\sharp,\cS}$  do not care the parity of the length of the expression as a regular continued fraction. In fact, the expression of odd length gives the same quivers. 
\end{rem}

\subsection{The basic properties of the left \textit{q}-deformed rational numbers}

The left $q$-deformed rational numbers satisfy the following basic properties, which correspond to \cite[Proposition 2.8]{LM} for the right variant, and the proof is also similar.

\begin{prop}\label{basic for flat}
    For $\frac{r}{s}\in \QQ, n\in \ZZ,$ one has
\begin{enumerate}
    \item $\displaystyle\left[\frac{r}{s}+n\right]^{\flat}_q=q^n\left[\frac{r}{s}\right]^{\flat}_q+\left[n\right]_q;$
    \item $\displaystyle\left[-\frac{r}{s}\right]^{\flat}_q=-q^{-1}\left[\frac{r}{s}\right]^{\flat}_{q^{-1}};$
    \item $\displaystyle\left[\frac{s}{r}\right]^{\flat}_q=\frac{1}{\left[\frac{r}{s}\right]^{\flat}_{q^{-1}}}.$
\end{enumerate}
\end{prop}

By Proposition~\ref{basic for flat} (1), for irreducible fractions $\frac{r}s, \frac{r'}s$ with $r \equiv r' \pmod{s}$, we have 
\begin{equation}\label{flat r equiv r'}
\flS_{\frac{r'}s}(q)= \flS_{\frac{r}s}(q).
\end{equation}

\begin{prop}\label{s+t=r}
The following hold. 
\begin{itemize}
\item[(1)] For irreducible fractions $\frac{r}s, \frac{r'}s$ with $r+r'\equiv 0 \pmod{s}$, we have $\flS_{\frac{r}s}(q) = \flS_{\frac{r'}s}(q)^\vee$. 
\item[(2)] For irreducible fractions $\frac{r}s, \frac{r}t >1$ with $s+t =r$,  we have $\flR_{\frac{r}s}(q) = \flR_{\frac{r}t}(q)^\vee$. 
\end{itemize}
\end{prop}

\begin{proof}
(1) By \eqref{flat r equiv r'},  we may assume that $ 1\le n < \frac{r}s < \frac{r'}s < n+1$. If $\frac{r}s=[n, a_2, \ldots, a_k]$, then $a_2 \ge 2$ and $\frac{r'}s=[n, 1, a_2-1, \ldots, a_k]$ by \cite[Lemma~3.1]{KMRWY}. So we have $Q_{\frac{r}{s}}^{\flat,\cS}=\left(Q_{\frac{r'}{s}}^{\flat,\cS}\right)^\vee$, and the assertion follows from Theorem~\ref{theorem:BBL closure} and \eqref{op-cl}. 

(2) By (1), we have $\flS_{\frac{s}r}(q)=\flS_{\frac{t}r}(q)^\vee$. By Proposition~\ref{basic for flat} (3), we have  $\flR_{\frac{r}s}(q) \equiv \flS_{\frac{s}r}(q^{-1})$ and $\flR_{\frac{r}t}(q) \equiv \flS_{\frac{t}r}(q^{-1})$. Since $\flR_{\frac{r}s}(q), \flR_{\frac{r}t}(q) \in \ZZ[q]$ with $\flR_{\frac{r}s}(0)=\flR_{\frac{r}t}(0)=1$ now, the assertion follows. 
\end{proof}

We remark that the right $q$-deformed rationals $\left[\dfrac{r}s\right]_q^\sharp$ also satisfy the equations corresponding to \eqref{flat r equiv r'} and Proposition~\ref{s+t=r} (c.f. \cite{MO1,KMRWY}).

\section{\textit{q}-transposes and orthogonal \textit{q}-transpose in \texorpdfstring{${\rm PSL}_q(2,\mathbb{Z})$}{PSLq(2,Z)}}\label{Sec3}

 In this section, we define two operations the {\it $q$-transpose} $A^{T_q}$ and the {\it orthogonal $q$-transpose} $A^{O_q},$ and give  some basic properties and applications of them.

\begin{defn}
For a matrix 

$$A=\begin{pmatrix}
\cR(q) & \mathcal{V}(q) \\
\cS(q) &  \mathcal{U}(q) \\
\end{pmatrix}$$
whose entries are elements in $\ZZ[q^{\pm 1}]$, set 
$$
A^{T_q}:=\begin{pmatrix}
\cR(q) & q^{-1}\cS(q) \\
q \mathcal{V}(q) &  \mathcal{U}(q) \\
\end{pmatrix}.
$$
We call it the {\it $q$-transpose} of $A$. 
\end{defn}

An easy calculation shows that $(A^{T_q})^{T_q} = A$, $\det (A^{T_q})= \det A$, $\Tr (A^{T_q})=\Tr A$, and $(AB)^{T_q}=B^{T_q}A^{T_q}$. The operation $(-)^{T_q}$ also makes sense up to the equivalence $\equiv$. 

\begin{lem}\label{q-transpose}
If $A \in \mathrm{PSL}_q(2,\ZZ)$, then $A^{T_q} \in \mathrm{PSL}_q(2,\ZZ)$. 
\end{lem}

\begin{proof}

We have $R_q^{T_q} \equiv L_q$ and $S_q^{T_q} \equiv S_q$. Moreover, if $A$ is regular, then $A^{T_q}$ is also, and 
$(A^{-1})^{T_q}\equiv (A^{T_q})^{-1}$. 
So the assertion follows.   
\end{proof}

The next result immediately follows from Proposition~\ref{column of a general element} and Lemma~\ref{q-transpose}. 

\begin{cor}\label{row of a general element} 

For 
$$\begin{pmatrix}
\cR(q) & \mathcal{V}(q) \\
\cS(q) &  \mathcal{U}(q) \\
\end{pmatrix} \in \mathrm{PSL}_{q}(2,\mathbb{Z})
$$   
with $r:=\cR(1)$, $s:=\cS(1)$, $v:=\mathcal{V}(1)$ and 
$u:=\mathcal{U}(1)$, we have  
$$
\left[\displaystyle\frac{r}{v}\right]^\sharp_q=\left[ \frac{\cR(q)}{q\cV(q)} \right] \quad  \text{and} \quad \left[\displaystyle\frac{s}{u}\right]^\sharp_q=\left [ \frac{\cS(q)}{q\cU(q)} \right ]. 
$$
\end{cor}

Now we can get the third proof of \cite[Theorem~3.5]{KMRWY} (\cite{KMRWY} already contains two proofs).

\begin{cor}[{\cite[Theorem~3.5]{KMRWY}}]\label{3rd proof}
For a positive integer $s$ and integers $v,w \in \ZZ$ with $vw \equiv -1 \pmod{s}$, we have $\shS_{\frac{v}s}(q) = \shS_{\frac{w}{s}}(q)$. 
\end{cor}

\begin{proof}
Since $vw \equiv -1 \pmod{s}$, 
there is some $r \in \ZZ$ such that $\begin{pmatrix}
r & v \\
w & s \\
\end{pmatrix} \in \mathrm{SL}(2,\mathbb{Z})$.
Now we can take $\cR(q),\cV(q),\cW(q),\cS(q) \in \ZZ[q^{\pm}]$ such that
$$
\begin{pmatrix}
\cR(q) & \cV(q) \\
\cW(q) & \cS(q) \\
\end{pmatrix} \in \mathrm{PSL}_q(2, \ZZ) 
\quad \text{and} \quad 
\begin{pmatrix}
\cR(1) & \cV(1) \\
\cW(1) & \cS(1) \\
\end{pmatrix}
\equiv \begin{pmatrix}
r & v \\
w & s \\
\end{pmatrix}.$$
By Proposition~\ref{column of a general element} and Corollary~\ref{row of a general element}, we have 
$$
\left[\displaystyle\frac{v}{s}\right]^\sharp_q
= \left[\frac{\shR_{\frac{v}{s}}(q)}{\shS_\frac{v}{s}(q)} \right]
= \left[\frac{\mathcal{V}(q)}{\mathcal{S}(q)} \right] \quad \text{and} \quad  \left[\displaystyle\frac{w}{s}\right]^\sharp_q
= \left[\frac{\shR_{\frac{w}{s}}(q)}{\shS_\frac{w}{s}(q)} \right]
= \left[
\frac{\mathcal{W}(q)}{q\mathcal{S}(q)}\right]. 
$$
Hence we have $\shS_\frac{v}{s}(q)\equiv \cS(q) \equiv \shS_\frac{w}{s}(q)$. Since $\shS_\frac{v}{s}(q), \shS_\frac{w}{s}(q) \in \ZZ[q]$ with $\shS_\frac{v}{s}(0)= \shS_\frac{w}{s}(0) =1$, we are done. 
\end{proof}

\begin{cor}[{\cite[Lemma~4.1]{KMRWY}}]\label{3rd proof for R}
For $r \in \ZZ$ and positive integers $v, \ w$ with $vw \equiv -1 \pmod{r}$, we have $\shR_{\frac{r}v}(q)\equiv\shR_{\frac{r}{w}}(q)$. If further $\frac{r}v, \frac{r}w>1$, we have   $\shR_{\frac{r}v}(q)=\shR_{\frac{r}{w}}(q)$.
\end{cor}

\begin{proof}
The first assertion can be shown in a similar way to Corollary~\ref{3rd proof}. However, we focus on the $(1,1)$-entry of the matrices in the proof of the corollary. Since $\shR_\alpha(q) \in \ZZ[q]$ with $\shR_\alpha(0)=1$ for $\alpha>1$, the second assertion follows from the first. 
\end{proof}

\begin{defn}
For a matrix 

$$A=\begin{pmatrix}
\cR(q) & \mathcal{V}(q) \\
\cS(q) &  \mathcal{U}(q) \\
\end{pmatrix}$$
whose entries are elements in $\ZZ[q^{\pm 1}]$, set 
$$
A^{O_q}:=\begin{pmatrix}
\mathcal{U}(q^{-1}) & q^{-1}\mathcal{V}(q^{-1}) \\
q \cS(q^{-1}) &  \cR(q^{-1}) \\
\end{pmatrix}.
$$
We call it the {\it orthogonal $q$-transpose} of $A$. 
\end{defn}

Easy calculation shows that $(A^{O_q})^{O_q} = A$, $(AB)^{O_q}=B^{O_q}A^{O_q}$. For $d(q):= \det A$ and $t(q):=\Tr A$, we have $\det (A^{O_q})=d(q^{-1})$ and $\Tr(A^{O_q})=t(q^{-1})$. 
Since the operation $(-)^{O_q}$  also makes  sense up to the equivalence $\equiv$, and we have $(R_q)^{O_q} \equiv R_q$ and $(S_q)^{O_q} \equiv S_q$. By the same way as Lemma~\ref{q-transpose}, we can show the following. 

\begin{lem}\label{orth & double q-transpose}
If $A \in \mathrm{PSL}_q(2,\ZZ)$, then $A^{O_q} \in \mathrm{PSL}_q(2,\ZZ)$. 
\end{lem}


\section{The trace of \texorpdfstring{$A \in \PSL_q(2,\ZZ)$ } {}
and the Jones polynomials of rational knots}\label{left}

In this section, we give further applications of the (orthogonal) \textit{q}-transpose. These applications include arithmetic properties of the left \textit{q}-deformed rationals, 
the trace of the matrices in $\PSL_q(2,\ZZ),$ and the normalized Jones polynomials of rational links.

We recall some basic notions and results on rational links and their normalized Jones polynomials. 

For an irreducible fraction $\frac{r}s=[a_1, \ldots, a_{2m}]>1$, the rational link associated with $\frac{r}s$ in the $3$-sphere 
$\mathbb{S}^3$ is determined by the Figure~\ref{L_alpha},
\begin{figure}[ht]
\begin{center}
\begin{tikzpicture}[x=0.6pt,y=0.4pt,yscale=-1,xscale=1]

\draw   (26,89) .. controls (26,84.58) and (29.58,81) .. (34,81) -- (88,81) .. controls (92.42,81) and (96,84.58) .. (96,89) -- (96,113) .. controls (96,117.42) and (92.42,121) .. (88,121) -- (34,121) .. controls (29.58,121) and (26,117.42) .. (26,113) -- cycle ;

\draw   (147,20) .. controls (147,15.58) and (150.58,12) .. (155,12) -- (209,12) .. controls (213.42,12) and (217,15.58) .. (217,20) -- (217,44) .. controls (217,48.42) and (213.42,52) .. (209,52) -- (155,52) .. controls (150.58,52) and (147,48.42) .. (147,44) -- cycle ;

\draw   (438,90) .. controls (438,85.58) and (441.58,82) .. (446,82) -- (500,82) .. controls (504.42,82) and (508,85.58) .. (508,90) -- (508,114) .. controls (508,118.42) and (504.42,122) .. (500,122) -- (446,122) .. controls (441.58,122) and (438,118.42) .. (438,114) -- cycle ;

\draw   (536,25) .. controls (536,20.58) and (539.58,17) .. (544,17) -- (598,17) .. controls (602.42,17) and (606,20.58) .. (606,25) -- (606,49) .. controls (606,53.42) and (602.42,57) .. (598,57) -- (544,57) .. controls (539.58,57) and (536,53.42) .. (536,49) -- cycle ;
 
\draw    (26,89) .. controls (-1,63) and (-20,-8) .. (147,20) ;

\draw    (26,113) .. controls (-10,124) and (3.5,181) .. (205.5,178) .. controls (407.5,175) and (618.5,185) .. (643.5,165) .. controls (668.5,145) and (700,3) .. (606,25) ;
 
\draw    (96,89) .. controls (128,59) and (107,74) .. (147,44) ;

\draw    (96,113) .. controls (134,114) and (248.5,113) .. (290.5,114) ;

\draw    (217,44) .. controls (250.5,36) and (249.5,87) .. (287.5,86) ;
 
\draw    (362.5,114) .. controls (404,114) and (403.5,115) .. (438,114) ;

\draw    (359.5,90) .. controls (401,90) and (403.5,91) .. (438,90) ;

\draw    (508,90) .. controls (515.5,80) and (517.5,65) .. (536,49) ;

\draw    (217,20) -- (271.5,22) ;
 
\draw    (368.5,24) -- (536,25) ;

\draw    (508,114) .. controls (655.5,175) and (646.5,83) .. (606,49) ;

\draw (45,88) node [anchor=north west][inner sep=0.75pt]    {$-a_{1}$};
\draw (175,25) node [anchor=north west][inner sep=0.75pt]    {$a_{2}$};
\draw (440,88) node [anchor=north west][inner sep=0.75pt]    {$-a_{2m-1}$};
\draw (559,28.4) node [anchor=north west][inner sep=0.75pt]    {$a_{2m}$};
\draw (310,108) node [anchor=north west][inner sep=0.75pt]  [font=\Large]  {$\cdots $};
\draw (310,80) node [anchor=north west][inner sep=0.75pt]  [font=\Large]  {$\cdots $};
\draw (310,18) node [anchor=north west][inner sep=0.75pt]  [font=\Large]  {$\cdots $};

\end{tikzpicture}
\caption{The rational link of $\frac{r}s.$}
   \label{L_alpha} 
\end{center}
\end{figure}
where each square is called $a_i$-half twists determined by Figure \ref{square}.
 
\begin{figure}[ht]
\begin{center}
\begin{tikzpicture}[x=0.5pt,y=0.3pt,yscale=-1,xscale=1]

\draw    (230.21,120.52) .. controls (291.63,119.6) and (265.32,16) .. (320.9,16.27) ;

\draw    (278.74,80.13) .. controls (289.08,126.39) and (309.39,121.4) .. (317.08,121.98) ;

\draw    (238.21,14.47) .. controls (259.72,14.98) and (261.06,29.01) .. (269.07,59.75) ;

\draw    (317.08,121.98) .. controls (378.5,121.06) and (337.48,-3.96) .. (393.67,18.34) ;

\draw    (360.86,81.93) .. controls (371.21,128.19) and (387.98,129.32) .. (395.35,118.19) ;

\draw    (320.34,16.27) .. controls (341.85,16.78) and (343.18,30.81) .. (351.2,61.55) ;
 
\draw    (441.56,117.63) .. controls (496.67,142.15) and (471.97,16.21) .. (527.55,16.48) ;

\draw    (488.23,83.6) .. controls (498.57,129.87) and (518.88,124.88) .. (526.57,125.45) ;

\draw    (444.48,17.31) .. controls (465.99,17.81) and (467.7,29.22) .. (475.72,59.96) ;

\draw    (312.23,402.7) .. controls (242.24,402.7) and (289.29,292.42) .. (229.3,289.01) ;
 
\draw    (262.24,356.09) .. controls (251.06,393.61) and (259.3,402.14) .. (229.89,402.14) ;
 
\draw    (310.47,289.01) .. controls (280.47,287.87) and (276.35,315.16) .. (271.06,332.21) ;
 
\draw    (404.87,435.58) -- (531.33,436.15) ;
\draw [shift={(531.33,436.15)}, rotate = 180.26] [color={rgb, 255:red, 0; green, 0; blue, 0 }  ][line width=0.75]    (0,5.59) -- (0,-5.59)(10.93,-3.29) .. controls (6.95,-1.4) and (3.31,-0.3) .. (0,0) .. controls (3.31,0.3) and (6.95,1.4) .. (10.93,3.29)   ;

\draw    (354.88,436.59) -- (228.42,436.84) ;
\draw [shift={(228.42,436.84)}, rotate = 359.89] [color={rgb, 255:red, 0; green, 0; blue, 0 }  ][line width=0.75]    (0,5.59) -- (0,-5.59)(10.93,-3.29) .. controls (6.95,-1.4) and (3.31,-0.3) .. (0,0) .. controls (3.31,0.3) and (6.95,1.4) .. (10.93,3.29)   ;

\draw    (392.81,402.7) .. controls (322.82,402.7) and (369.88,292.42) .. (309.88,289.01) ;
 
\draw    (342.82,356.09) .. controls (331.64,393.61) and (341.64,402.7) .. (312.23,402.7) ;
 
\draw    (391.05,289.01) .. controls (361.05,287.87) and (356.94,315.16) .. (351.64,332.21) ;

\draw    (529.86,401.57) .. controls (459.87,401.57) and (506.92,291.28) .. (446.93,287.87) ;

\draw    (479.87,354.95) .. controls (468.69,392.47) and (476.93,401) .. (447.52,401) ;

\draw    (528.1,287.87) .. controls (498.1,286.74) and (493.98,314.02) .. (488.69,331.08) ;
 
\draw   (10,48.8) .. controls (10,41.73) and (15.73,36) .. (22.8,36) -- (106.2,36) .. controls (113.27,36) and (119,41.73) .. (119,48.8) -- (119,87.2) .. controls (119,94.27) and (113.27,100) .. (106.2,100) -- (22.8,100) .. controls (15.73,100) and (10,94.27) .. (10,87.2) -- cycle ;
 
\draw   (10.67,331.47) .. controls (10.67,324.4) and (16.4,318.67) .. (23.47,318.67) -- (106.87,318.67) .. controls (113.94,318.67) and (119.67,324.4) .. (119.67,331.47) -- (119.67,369.87) .. controls (119.67,376.94) and (113.94,382.67) .. (106.87,382.67) -- (23.47,382.67) .. controls (16.4,382.67) and (10.67,376.94) .. (10.67,369.87) -- cycle ;

\draw    (405.21,151.26) -- (529.67,151.81) ;
\draw [shift={(529.67,151.81)}, rotate = 180.25] [color={rgb, 255:red, 0; green, 0; blue, 0 }  ][line width=0.75]    (0,5.59) -- (0,-5.59)(10.93,-3.29) .. controls (6.95,-1.4) and (3.31,-0.3) .. (0,0) .. controls (3.31,0.3) and (6.95,1.4) .. (10.93,3.29)   ;

\draw    (356.02,152.24) -- (231.56,152.47) ;
\draw [shift={(231.56,152.47)}, rotate = 359.89] [color={rgb, 255:red, 0; green, 0; blue, 0 }  ][line width=0.75]    (0,5.59) -- (0,-5.59)(10.93,-3.29) .. controls (6.95,-1.4) and (3.31,-0.3) .. (0,0) .. controls (3.31,0.3) and (6.95,1.4) .. (10.93,3.29)   ;

\draw (408.29,108.56) node [anchor=north west][inner sep=0.75pt]  [rotate=-359.13]  {$\cdots $};
\draw (404.27,10.81) node [anchor=north west][inner sep=0.75pt]  [rotate=-359.13]  {$\cdots $};
\draw (374.55,426.63) node [anchor=north west][inner sep=0.75pt]    {$a_{i}$};
\draw (406.1,392.31) node [anchor=north west][inner sep=0.75pt]  [rotate=-359.13]  {$\cdots $};
\draw (402.46,280.65) node [anchor=north west][inner sep=0.75pt]  [rotate=-359.13]  {$\cdots $};
\draw (40,45) node [anchor=north west][inner sep=0.75pt]  [font=\Large]  {$-a_{i}$};
\draw (153,48.4) node [anchor=north west][inner sep=0.75pt]  [font=\Large]  {$=$};
\draw (55.11,334.47) node [anchor=north west][inner sep=0.75pt]  [font=\Large]  {$a_{i}$};
\draw (153.67,331.07) node [anchor=north west][inner sep=0.75pt]  [font=\Large]  {$=$};
\draw (375.27,142.33) node [anchor=north west][inner sep=0.75pt]    {$a_{i}$};

\end{tikzpicture}
\caption{$a_i$-half twists}
   \label{square} 
\end{center}
\end{figure}
We denote by $L(\frac{r}s)$ the rational link associated with $\frac{r}s$. See, for example, \cite{KL,KR}. The following theorem is due to H. Schubert in 1956.

\begin{thm}[ {c.f. \cite[Theorem~2]{KL}}]\label{Schubert}
For irreducible fractions $\frac{r}s, ~\frac{r'}{s'}$, the following are equivalent.
\begin{itemize}
\item[(1)] $L(\frac{r}s)$ and $L(\frac{r'}{s'})$ are isotopic. 
\item[(2)] $r=r'$ and either $s \equiv s' \pmod{r}$ or $ss' \equiv 1 \pmod{r}$. 
\end{itemize}
\end{thm}

As a useful isotopy invariant for an oriented link $L$ in $\mathbb{S}^3$, the Jones polynomial $V_L(t)~\in \ZZ[t^{\pm 1}]\cup t^{\frac12}\ZZ[t^{\pm 1}]$ is well-studied. 
Lee and Schiffler~\cite[Proposition 1.2 (b)]{KR} introduced the following normalization $J_{\alpha }(q)$ of the Jones polynomial $V_{L(\alpha )}(t)$ of a rational link $L(\alpha )$: 
\begin{equation}\label{def normalized Jones}
 J_{\alpha }(q):=\pm t^{-h}V_{L(\alpha)}(t)|_{t=-q^{-1}},
\end{equation}
where $\pm t^{h}$ is the leading term of $V_{L(\alpha)}(t)$. 
This indicates the normalization such that the constant term is $1$ as a polynomial in $q$.

\medskip

\begin{thm}[{ \cite[Proposition A.1]{MO1}}]\label{Jones}
For a rational number $\alpha >1$,  the normalized Jones polynomial $J_{\alpha}(q)$ can be computed by 
\begin{equation}\label{eq2-6}
J_{\alpha }(q)=q\shR_{\alpha }(q)+(1-q)\shS_{\alpha }(q).  
\end{equation} 
\end{thm}

 Now, we can give a new proof of the following result. This result was shown by Bapat, Becker and Licata \cite[Theorem~A3]{BBL} by considering a homological interpretation of the left $q$-deformed rational numbers, and the first author of the present paper gave a combinatorial proof by using $q$-deformed Farey sums and induction on the size of negative continued fractions (c.f \cite[Theorem~4.2]{XR2}). We remark that all known proofs (including ours) use Theorem~\ref{Jones}.

\begin{thm}[c.f. {\cite[Theorem~A3]{BBL}, see also \cite[Theorem~4.2]{XR2}}]\label{J-L}
For a rational number $\alpha>1$, we have 
$$J_\alpha(q)=\flR_\alpha (q)^\vee.$$
\end{thm}

\begin{proof}
We can take an irreducible fraction $\frac{r}s >1$ with $\alpha=\frac{r}{r-s}$.  
If $\frac{r}s=[a_1, \ldots, a_{2m}]$, we have 
$$M(a_1, \ldots, a_{2m})=\begin{pmatrix}
r & t \\
s & u \\
\end{pmatrix}$$
for some $u \in \NN$, and 
$$M_q(a_1, \ldots, a_{2m}) = 
\begin{pmatrix}
q\shR_{\frac{r}s}(q) & \shR_{\frac{t}u}(q) \\
q \shS_{\frac{r}s}(q) &  \shS_{\frac{t}u}(q)\\
\end{pmatrix}.$$
By Corollaries~\ref{3rd proof for R} and \ref{row of a general element}, we have $\shR_{\frac{r}s}(q) = \shR_{\frac{r}t}(q)$ and $\shR_{\frac{t}u}(q) = \shS_{\frac{r}t}(q)$. Hence we have  
$$\flR_{\frac{r}s}(q) = q\shR_{\frac{r}s}(q) + (1-q) \shR_{\frac{t}u}(q)
=q\shR_{\frac{r}t}(q) + (1-q)\shS_{\frac{r}t}(q)
=J_{\frac{r}t}(q),$$
where the last equality follows from Theorem~\ref{Jones}. 
Since $t(r-s) \equiv -st \equiv 1 \pmod{r}$, we have $J_{\frac{r}t}(q)=J_{\frac{r}{r-s}}(q)=J_\alpha(q)$ by Theorem~\ref{Schubert}. By Proposition~\ref{s+t=r}~(2), we have $\flR_{\frac{r}s}(q)=\flR_{\frac{r}{r-s}}(q)^\vee$. Summing up, we have 
$$\flR_\alpha(q)^\vee = \flR_{\frac{r}{r-s}}(q)^\vee=\flR_{\frac{r}s}(q)=J_{\frac{r}t}(q)=J_\alpha(q).$$
\end{proof}

\begin{lem}\label{arithmetic of flat lem} 
 For irreducible fractions $\frac{r}s, \frac{r}{s'} >1$ with $ss' \equiv 1 \pmod{r}$ {\rm (}resp. $ss' \equiv -1 \pmod{r}${\rm )},  we have $\flR_{\frac{r}s}(q) = \flR_{\frac{r}{s'}}(q)$ {\rm (}resp. $\flR_{\frac{r}s}(q) = \flR_{\frac{r}{s'}}(q)^\vee${\rm )}.   
\end{lem}

\begin{proof}
If $ss' \equiv 1 \pmod{r}$, then the links $L(\frac{r}s)$ and $L(\frac{r}{s'})$ are isotopic by Theorem~\ref{Schubert}, and have the same (normalized) Jones polynomial. By Theorem~\ref{Jones}, we have  
$$\flR_{\frac{r}s}(q) =J_{\frac{r}s}(q)^\vee=J_{\frac{r}{s'}}(q)^\vee=\flR_{\frac{r}{s'}}(q).$$
The case $ss' \equiv -1 \pmod{r}$ follows from the above equation and Proposition~\ref{s+t=r}~(2). 
\end{proof}

\begin{rem}

If $ss' \equiv -1 \pmod{r}$, then it is known that $\frac{r}s=[a_1, \ldots, a_{2m}]$ implies $\frac{r}{s'}=[a_{2m}, \ldots, a_1]$ $($to see this,  take the transpose of $M(a_1, \ldots, a_{2m})$$)$. Hence, the $ss' \equiv -1 \pmod{r}$ case of the above lemma states that, for a tuple of integers $(b_1,b_2,\ldots, b_{2m})$ with $b_1, b_{2m}\geq 0$, $b_2,\ldots,b_{2m-1} >0$, the equation 
\begin{equation}\label{property of Q^flat}
\cl(Q^\flat(b_1, \ldots, b_{2m}),q) =\cl(Q^\flat(b_{2m}, \ldots, b_1),q)^\vee
\end{equation}
holds. It is easy to see that \eqref{property of Q^flat} is equivalent to the following equation for $(b_1,\ldots, b_{2m+1})$ with $b_1, b_{2m+1}\geq 0$, $b_2,\ldots,b_{2m} >0$:
\begin{equation}\label{flat quiver of odd length}
\cl(Q^\flat(b_1, \ldots, b_{2m+1}),q) =\cl(Q^\flat(b_{2m+1}, \ldots, b_1),q).
\end{equation}
It might be an interesting problem to find a bijective proof of the above equations. 
\end{rem}

\
\begin{exam}
This is an example of the equation \eqref{flat quiver of odd length}. Considering two tuples of integers $(1,2,0)$ and $(0,2,1),$ we can draw $Q^\flat(1,2,0)$ (for convenience, we number each vertex) and count its $l$-closures as follows.

\begin{figure}[ht]
\begin{center}
\begin{tikzpicture}[x=0.6pt,y=0.5pt,yscale=-1,xscale=1]

\draw    (74.54,42.83) -- (34.71,43.58) ;
\draw [shift={(32.71,43.62)}, rotate = 358.92] [color={rgb, 255:red, 0; green, 0; blue, 0 }  ][line width=0.75]    (10.93,-3.29) .. controls (6.95,-1.4) and (3.31,-0.3) .. (0,0) .. controls (3.31,0.3) and (6.95,1.4) .. (10.93,3.29)   ;

\draw    (167.18,44.13) -- (207.01,43.42) ;
\draw [shift={(209.01,43.38)}, rotate = 178.98] [color={rgb, 255:red, 0; green, 0; blue, 0 }  ][line width=0.75]    (10.93,-3.29) .. controls (6.95,-1.4) and (3.31,-0.3) .. (0,0) .. controls (3.31,0.3) and (6.95,1.4) .. (10.93,3.29)   ;

\draw    (100.2,43.98) -- (140.03,43.27) ;
\draw [shift={(142.03,43.23)}, rotate = 178.98] [color={rgb, 255:red, 0; green, 0; blue, 0 }  ][line width=0.75]    (10.93,-3.29) .. controls (6.95,-1.4) and (3.31,-0.3) .. (0,0) .. controls (3.31,0.3) and (6.95,1.4) .. (10.93,3.29)   ;

\draw    (231.86,36.76) .. controls (252.7,-5.83) and (273.41,12.07) .. (285.31,34.7) ;
\draw [shift={(286.21,36.45)}, rotate = 243.39] [color={rgb, 255:red, 0; green, 0; blue, 0 }  ][line width=0.75]    (10.93,-3.29) .. controls (6.95,-1.4) and (3.31,-0.3) .. (0,0) .. controls (3.31,0.3) and (6.95,1.4) .. (10.93,3.29)   ;

\draw    (284.94,49.63) .. controls (262.9,76.46) and (245.25,72.41) .. (234.22,49.84) ;
\draw [shift={(233.38,48.07)}, rotate = 65.55] [color={rgb, 255:red, 0; green, 0; blue, 0 }  ][line width=0.75]    (10.93,-3.29) .. controls (6.95,-1.4) and (3.31,-0.3) .. (0,0) .. controls (3.31,0.3) and (6.95,1.4) .. (10.93,3.29)   ;

\draw (10.12,40) node [anchor=north west][inner sep=0.75pt]  [font=\Large]  {$\circ $};
\draw (76.88,40) node [anchor=north west][inner sep=0.75pt]  [font=\Large]  {$\circ $};
\draw (143.93,40) node [anchor=north west][inner sep=0.75pt]  [font=\Large]  {$\circ $};
\draw (211.97,40) node [anchor=north west][inner sep=0.75pt]  [font=\Large]  {$\circ $};
\draw (287.26,40) node [anchor=north west][inner sep=0.75pt]  [font=\Large]  {$\circ $};
\draw (12,20) node [anchor=north west][inner sep=0.75pt]    {$1$};
\draw (78,20) node [anchor=north west][inner sep=0.75pt]    {$2$};
\draw (145,20) node [anchor=north west][inner sep=0.75pt]    {$3$};
\draw (212,20) node [anchor=north west][inner sep=0.75pt]    {$4$};
\draw (290,20) node [anchor=north west][inner sep=0.75pt]    {$5$};
\end{tikzpicture}
\end{center}
\caption{The quiver $Q^\flat(1,2,0).$}
   \label{quiver_(120)} 
\end{figure}

\begin{table}[ht]
 \centering
  \begin{tabular}{c@{\hspace{3.0cm}}ll}
   \hline \hline
 $l$  &  $l$-closures &  the number of  $l$-closures
\\
   \hline \hline
  $0$  &  $\emptyset$ &\qquad \qquad $1$  \\
   \hline
   $1$  &  $\{1\}$ & \qquad \qquad$1$  \\ 
    \hline
    $2$  &  $\{4,5\}$ & \qquad \qquad$1$  \\
    \hline
     $3$  &  $\{1,4,5\}, \{3,4,5\}$ & \qquad \qquad$2$  \\
    \hline
       $4$  &  $\{1,3,4,5\}$ & \qquad \qquad$1$  \\
    \hline

       $5$  &  $\{1,2,3,4,5\}$ & \qquad \qquad$1$  \\
    \hline
  \end{tabular}\label{table:001}
  \caption{The $l$-closure of $Q^\flat(1,2,0)$}
\end{table}
Hence we have
$$\cl(Q^\flat(1,2,0))=1+q+q^2+2q^3+q^4+q^5.$$

On the other hand, 
we can draw $Q^\flat(0,2,1)$, and count its $l$-closures as follows. 
\begin{figure}[ht]
\begin{center}
\begin{tikzpicture}[x=0.6pt,y=0.5pt,yscale=-1,xscale=1]

\draw    (211.34,43.23) -- (171.51,43.98) ;
\draw [shift={(169.51,44.02)}, rotate = 358.92] [color={rgb, 255:red, 0; green, 0; blue, 0 }  ][line width=0.75]    (10.93,-3.29) .. controls (6.95,-1.4) and (3.31,-0.3) .. (0,0) .. controls (3.31,0.3) and (6.95,1.4) .. (10.93,3.29)   ;

\draw    (33.18,44.13) -- (73.01,43.42) ;
\draw [shift={(75.01,43.38)}, rotate = 178.98] [color={rgb, 255:red, 0; green, 0; blue, 0 }  ][line width=0.75]    (10.93,-3.29) .. controls (6.95,-1.4) and (3.31,-0.3) .. (0,0) .. controls (3.31,0.3) and (6.95,1.4) .. (10.93,3.29)   ;

\draw    (100.2,43.98) -- (140.03,43.27) ;
\draw [shift={(142.03,43.23)}, rotate = 178.98] [color={rgb, 255:red, 0; green, 0; blue, 0 }  ][line width=0.75]    (10.93,-3.29) .. controls (6.95,-1.4) and (3.31,-0.3) .. (0,0) .. controls (3.31,0.3) and (6.95,1.4) .. (10.93,3.29)   ;

\draw    (231.86,36.76) .. controls (252.7,-5.83) and (273.41,12.07) .. (285.31,34.7) ;
\draw [shift={(286.21,36.45)}, rotate = 243.39] [color={rgb, 255:red, 0; green, 0; blue, 0 }  ][line width=0.75]    (10.93,-3.29) .. controls (6.95,-1.4) and (3.31,-0.3) .. (0,0) .. controls (3.31,0.3) and (6.95,1.4) .. (10.93,3.29)   ;
 
\draw    (284.94,49.63) .. controls (262.9,76.46) and (245.25,72.41) .. (234.22,49.84) ;
\draw [shift={(233.38,48.07)}, rotate = 65.55] [color={rgb, 255:red, 0; green, 0; blue, 0 }  ][line width=0.75]    (10.93,-3.29) .. controls (6.95,-1.4) and (3.31,-0.3) .. (0,0) .. controls (3.31,0.3) and (6.95,1.4) .. (10.93,3.29)   ;

\draw (10.12,40) node [anchor=north west][inner sep=0.75pt]  [font=\Large]  {$\circ $};
\draw (76.88,40) node [anchor=north west][inner sep=0.75pt]  [font=\Large]  {$\circ $};
\draw (143.93,40) node [anchor=north west][inner sep=0.75pt]  [font=\Large]  {$\circ $};
\draw (211.97,40) node [anchor=north west][inner sep=0.75pt]  [font=\Large]  {$\circ $};
\draw (287.26,40) node [anchor=north west][inner sep=0.75pt]  [font=\Large]  {$\circ $};
\draw (12,20) node [anchor=north west][inner sep=0.75pt]    {$1$};
\draw (78,20) node [anchor=north west][inner sep=0.75pt]    {$2$};
\draw (145,20) node [anchor=north west][inner sep=0.75pt]    {$3$};
\draw (212,20) node [anchor=north west][inner sep=0.75pt]    {$4$};
\draw (290,20) node [anchor=north west][inner sep=0.75pt]    {$5$};
\end{tikzpicture}
\end{center}
\caption{The quiver $Q^\flat(0,2,1).$}
   \label{quiver_(021)} 
\end{figure}
\begin{table}[ht]
 \centering
  \begin{tabular}{c@{\hspace{3.0cm}}ll}
   \hline \hline
  $l$  &  $l$-closures & the number of  $l$-closures\\ 
   \hline \hline
  $0$  &  $\emptyset$ &\qquad \qquad $1$  \\
   \hline
   $1$  &  $\{3\}$ & \qquad \qquad$1$  \\ 
    \hline
    $2$  &  $\{2,3\}$ & \qquad \qquad$1$  \\
    \hline
     $3$  &  $\{1,2,3\}, \{3,4,5\}$ &\qquad \qquad $2$  \\
    \hline
        $4$  &  $\{2,3,4,5\}$ & \qquad \qquad$1$  \\
    \hline
       $5$  &  $\{1,2,3,4,5\}$ & \qquad \qquad$1$  \\
    \hline
  \end{tabular}
\caption{The $l$-closure of $Q^\flat(0,2,1)$} 
\label{table:002}
\end{table}

Hence we have
$$\cl(Q^\flat(0,2,1))=1+q+q^2+2q^3+q^4+q^5=\cl(Q^\flat(1,2,0)).$$
\end{exam}

\medskip

\begin{thm}\label{arithmetic of flat} 
The following hold. 
\begin{itemize}
\item[(1)]  For irreducible fractions $\frac{r}s, \frac{r'}s$ with $rr' \equiv 1 \pmod{s}$ $($resp. $rr' \equiv -1 \pmod{s}$~$)$,  we have $\flS_{\frac{r}s}(q)=\flS_{\frac{r^\prime}{s}}(q)$ {\rm (}resp. $\flS_{\frac{r}s}(q)=\flS_{\frac{r^\prime}{s}}(q)^\vee$~$)$.
\item[(2)]  For irreducible fractions $\frac{r}s, \frac{r}{s'}$ with $ss' \equiv 1 \pmod{r}$ \rm{(}resp. $ss' \equiv -1 \pmod{r}$\rm{)},  we have $\flR_{\frac{r}s}(q) \equiv \flR_{\frac{r}{s'}}(q)$ \rm{(}resp. $\flR_{\frac{r}s}(q) \equiv \flR_{\frac{r}{s'}}(q)^\vee$ \rm{)}.   
(That is, if we replace $=$ by $\equiv$,  the statement of Lemma~\ref{arithmetic of flat lem} holds without the restriction that $\frac{r}s, \frac{r}{s'} >1$.)
\end{itemize}
\end{thm}
\begin{proof}

(1) By \eqref{flat r equiv r'}, we may assume that $0 < \frac{r}s, \frac{r'}s < 1$. Since we always have $\flS_{\frac{r}s}(q) \in \ZZ[q]$ with $\flS_{\frac{r}s}(0)=1$, the assertion follows from Lemma~\ref{arithmetic of flat lem} and 
Proposition~\ref{basic for flat} (3). 

(2) The assertion follows from (1) and Proposition~\ref{basic for flat} (3). 
\end{proof}

Using the above property of $\flR_{\frac{r}s}(q)$ and (orthogonal) $q$-transposes, we can get a new proof of the following impressive result on $\Tr A$ for $A \in \PSL_q(2,\ZZ)$.

\begin{thm}[{\cite[Theorem~3]{LM}}]\label{Tr A}
For $A \in \PSL_q(2,\ZZ)$, $\Tr A$ is equivalent to a palindromic polynomial whose coefficients are non-negative. 
\end{thm}

In the following, we say $f(q) \in \ZZ[q]$ is {\it anti-palindromic} if $f(q) = -f(q)^\vee$. 
\begin{proof}
In this proof, we treat  
$$M:=(-S_q)^{-1}A=\begin{pmatrix}
\cR(q) &   \cT(q)  \\
\cS(q) & \cU(q) \\
\end{pmatrix} \quad \text{with} \quad M|_{q=1}=\begin{pmatrix}
r & t \\
s & u
\end{pmatrix}
$$ 
rather than $A$ itself. Note that 
$$
M \equiv 
\begin{pmatrix}
q \shR_{\frac{r}s}(q) & \shS_{\frac{r}t}(q) \\
q \shS_{\frac{r}s}(q) & \cU(q) \\
\end{pmatrix}
\equiv 
\begin{pmatrix} 
\cR(q) & \shR_{\frac{t}u}(q) \\
q \shR_{\frac{s}u}(q) & \shS_{\frac{t}u}(q) \\
\end{pmatrix}
$$
by Proposition~\ref{column of a general element} and Corollary~\ref{row of a general element},  
$$A = (-S_q)M \equiv\begin{pmatrix}
\shS_{\frac{r}s}(q) &   q^{-1}\cU(q)  \\
-q\shR_{\frac{r}s}(q) & -\shS_{\frac{r}t}(q) \\
\end{pmatrix}
\equiv
\begin{pmatrix}
\shR_{\frac{s}u}(q) &   q^{-1}\shS_{\frac{t}u}(q)  \\
-\cR(q) & -\shR_{\frac{t}u}(q) \\
\end{pmatrix}$$
and $$\Tr A\equiv  \shS_{\frac{r}s}(q)-\shS_{\frac{r}t}(q)\equiv 
\shR_{\frac{s}u}(q)-\shR_{\frac{t}u}(q).$$
So we may assume that $s \ne t$. 

First, consider the case $u=0$. 
Then we have $(s,t)=(\pm 1,\mp 1)$, and  
$(\cS(q),\cT(q)) \equiv (q^n, -1)$ for some $n \in \ZZ$. 
Therefore $\Tr A \equiv \cS(q)-\cT(q)=q^n+1$. 
 So we may assume that $u \ne 0$. 

Replacing $A$ by either $-A$, $A^{O_q}$ or $-A^{O_q}$ , if necessary, 
we may assume that $u>0$ and $s > t$. By Proposition~\ref{basic for flat} (1) and Corollary~\ref{3rd proof} (note that $st \equiv -1 \pmod{u}$ now), we have 
\begin{align*}
\shR_{\frac{s}u+n}(q)-\shR_{\frac{t}{u}+n}(q) &= (q^n\shR_{\frac{s}u}(q)+[n]_q \cdot \shS_{\frac{s}u}(q))-  (q^n \shR_{\frac{t}u}(q)+[n]_q \cdot \shS_{\frac{t}u}(q)) \\
&= (q^n\shR_{\frac{s}u}(q)+[n]_q \cdot \shS_{\frac{s}u}(q))-  (q^n \shR_{\frac{t}u}(q)+[n]_q \cdot \shS_{\frac{s}u}(q)) \\
&= q^n(\shR_{\frac{s}u}(q) - \shR_{\frac{t}u}(q)) \equiv \Tr A. 
\end{align*}
We also remark that there is some $\cR'(q)\in \ZZ[q^{\pm 1}]$ such that 
$$
\begin{pmatrix}
\cR'(q) & \shR_{\frac{t}{u}+n}(q) \\
q \shR_{\frac{s}u+n}(q) & \cU(q) \\
\end{pmatrix} 
 \in \PSL_q(2,\ZZ). 
$$

So keeping $\Tr A$ up to $\equiv$, we may assume that $s,t > u>0$. 
More precisely, we replace $s$ (resp. $t$) by $s+nu$ (resp. $t+nu$)  for $n \gg 0$, if necessary. 
Then $r > s,t>u$, and 
$$M_q(a_1, \ldots, a_{2m})= 
\begin{pmatrix}
q \shR_{\frac{r}s}(q) & \shS_{\frac{r}t}(q) \\
q \shS_{\frac{r}s}(q) & \shS_{\frac{t}u}(q) \\
\end{pmatrix}$$ 
for $\frac{r}s=[a_1, \ldots, a_{2m}]>1$ by Lemma~\ref{reduced form}. 
Then $\frac{r}t=[a_{2m}, \ldots, a_1]>1$ and 
$$
M_q(a_{2m}, \ldots, a_1)= 
\begin{pmatrix}
q \shR_{\frac{r}s}(q) & \shS_{\frac{r}s}(q) \\
q \shS_{\frac{r}t}(q) & \shS_{\frac{t}{u}}(q) \\
\end{pmatrix}.  
$$
Hence
$$\flR_{\frac{r}s}(q)=q\shR_{\frac{r}s}+(1-q)\shS_{\frac{r}t} \quad \text{and} \quad \flR_{\frac{r}t}(q)=q\shR_{\frac{r}s}+(1-q)\shS_{\frac{r}s}.$$
Since $\flR_{\frac{r}s}(q) = \flR_{\frac{r}t}(q)^\vee$ by Lemma~\ref{arithmetic of flat} (note that $st \equiv -1 \pmod{r}$ now), 
\begin{equation}\label{s-t}
\flR_{\frac{r}s}(q) -\flR_{\frac{r}t}(q)=(q-1)(\shS_{\frac{r}s}(q) -\shS_{\frac{r}t}(q))
\end{equation}
is equivalent to an anti-palindromic polynomial. Since $(q-1)$ is also anti-palindromic,  $\shS_{\frac{r}s}(q) -\shS_{\frac{r}t}(q) \, (\equiv \Tr A)$ is  equivalent to a palindromic polynomial. 

Next we will show that all  non-zero coefficients of $\Tr A$ have the same sign. 
Using the above equation $\shR_{\frac{s}u+n}(q)-\shR_{\frac{t}{u}+n}(q) \equiv \shR_{\frac{s}u}(q)-\shR_{\frac{t}{u}}(q)$ for all $n \in \ZZ$, we may assume that $0 \ge t > -u$. Then all non-zero coefficients of $\shR_{\frac{t}u}(q)$ are negative. So if $s \ge 0$, all non-zero coefficients of $\shR_{\frac{s}u}(q)$ are positive, and the same is true for $\shR_{\frac{s}u}(q)-\shR_{\frac{t}{u}}(q) 
\, (\equiv \Tr A)$. 

It remains to show the case $0 > s > t> -u$. In this case, 
$|s|, |t| < u$. Since $ru-st=\pm 1$ now, we have $u > r> 0$.  
Since 
$$M^{O_q} \equiv 
\begin{pmatrix}
 \shS_{\frac{t}u}(q^{-1}) & q^{-1}\shR_{\frac{t}u}(q^{-1}) \\
\shR_{\frac{s}u}(q^{-1}) &  \cR(q^{-1})\\
\end{pmatrix}
\equiv \begin{pmatrix}
\cU'(q) &   \shR_{\frac{t}r}(q)  \\
q \shR_{\frac{s}r}(q) & \shS_{\frac{t}r}(q)  \\
\end{pmatrix} \text{ with }  M^{O_q}|_{q=1} = \begin{pmatrix}
u &   t \\
s &  r   \\
\end{pmatrix}, $$
if we set $\tau(q):=\Tr A$ then 
we have $$\tau(q^{-1})\equiv \shR_{\frac{s}u}(q^{-1})-\shR_{\frac{t}{u}}(q^{-1}) \equiv \shR_{\frac{s}r}(q)-\shR_{\frac{t}{r}}(q).$$ 
In this way, it suffices to show that all non-zero coefficients of $\shR_{\frac{s}r}(q)-\shR_{\frac{t}{r}}(q)$ have the same sign. So we  apply the above argument to $M^{O_q}.$ That is,  the reduction using $$\shR_{\frac{s}r+n}(q)-\shR_{\frac{t}{r}+n}(q) \equiv \shR_{\frac{s}r}(q)-\shR_{\frac{t}r}(q).$$  Next we take the orthogonal $q$-transpose again, if necessary. Repeating this procedure, we can reduce to the case $s > 0 > t$. 
\end{proof}

\begin{prop}\label{3 types for trace}  
For  $A \in \mathrm{PSL}_q(2,\ZZ)$, the trace $\Tr A$ is equivalent to one of the following polynomials.  
\begin{itemize}
\item[(1)] $1+q^n$ for some $n \in \ZZ_{\ge 0}$. 
\item[(2)] $[n]_q$ for some $n \in \ZZ_{\ge 0}$ $($note that $[0]_q=0$$)$. 
\item[(3)] 
$\Tr M_q(a_1, \ldots, a_{2m})$ for $a_1, \ldots, a_{2m} >0$. 
\end{itemize} 
\end{prop}

\begin{proof}
We reuse the proof of Theorem~\ref{Tr A}. 
In the notation there, the case $u=0$ corresponds to (1). If $s=t$, then $\Tr A=0$. 
So, in that proof, assuming $u \ne 0$ and $s \ne t$, we reduced to the case
$$
M|_{q=1}=\begin{pmatrix}
r&t \\ s & u 
\end{pmatrix} \quad \text{and} \quad 
A^{T_q}|_{q=1}=\begin{pmatrix}
s&-r \\ u & -t 
\end{pmatrix}. 
$$
with $0 \ge t>-u$ and $s \ge 0$. If $s=0$ or $t=0$, then $\Tr A \equiv [n]_q$ for some $n \in \NN$. So it remains to consider the case $-u <t<0<s$. Now  we have  
$A^{T_q}=M_q(a_1, \ldots, a_{2m})$ for some $a_1 \ge 0$ and 
$a_2, \ldots, a_{2m}>0$ by Lemma~\ref{reduced form}. If $s>u$, then we have $a_1>1$ and $\Tr A = \Tr A^{T_q}$ is the type (3). 
So assume that $0< s<u$. Since $|s|, |t|<u$ now, we have $0<|r|<u$.  
Like the last step of the proof of Theorem~\ref{Tr A}, we replace $M$ by $M^{O_q}$ keeping $\Tr A$ up to $\equiv$.  And then we ``adjust" $M$ again to satisfy $0 \ge t>-u$ and $s \ge 0$. By repeating this procedure if necessary, we can reduce to the case $s >u$. 
\end{proof}

For a sequence of positive integers $a_1, \ldots, a_{2m}$, \Oguz and Ravichandran \cite{OR} defined the {\it circular fence poset} $\overline{Q}(a_1, \ldots, a_{2m})$, and \Oguz \cite{KO25} showed that 
$$\Tr (M_q(a_1, \ldots, a_{2m}))=\mathsf{rk}(\overline{Q}(a_1, \ldots, a_{2m}), q)$$
(\cite[Proposition~5.5]{KO25}), and hence 
$$\Tr (M_q(a_1, \ldots, a_{2m}))|_{q=0}=\mathsf{rk}(\overline{Q}(a_1, \ldots, a_{2m}), 0)=1.$$
She also raised the following conjecture. 
For a polynomial $f(q)=\sum_{i=0}^n a_i q^i \in \ZZ[q]$, we say $f(q)$ is {\it unimodal} if 
$a_0 \le \cdots \le a_m  \ge \cdots \ge a_n
$ for some $m$ with $0 \le m \le n$. 

\begin{conj}[{\cite[Conjecture~1.4]{KO25}}]\label{Oguz}
For any sequence $a_1, \ldots, a_{2m} >0$, the trace  $\Tr (M(a_1, \ldots, a_{2m}))$ is unimodal except for the cases  
$(a_1,a_2, \ldots, a_{2m})=(1,k,1,k)$ or $(k,1,k,1)$ for some $k$. ($m=2$ holds in these cases.)    
\end{conj}

We use the above observation to analyze the behavior of $\flR_{\frac{r}s}(q)$. 

\begin{lem}\label{Palin for flR}
For an irreducible fraction $\frac{r}s >1$, $\flR_{\frac{r}s}(q)$ is palindromic if and only if $s^2 \equiv -1 \pmod{r}$. 
\end{lem}

\begin{proof}
Set $\frac{r}s=[a_1, \ldots, a_{2m}]$ and $M(a_1, \ldots, a_{2m})=\begin{pmatrix}r& t\\s& u \end{pmatrix}$. Then we have  
$$\flR_{\frac{r}s}(q)-\flR_{\frac{r}s}(q)^\vee
=\flR_{\frac{r}s}(q)-\flR_{\frac{r}t}(q)=(q-1)(\shS_{\frac{r}s}(q) -\shS_{\frac{r}t}(q))$$
by \eqref{s-t}. 
Since $\shS_{\frac{r}s}(1)=s$ and $\shS_{\frac{r}t}(1)=t$, $\shS_{\frac{r}s}(q) -\shS_{\frac{r}t}(q)=0$ if and only if $s=t$.  The former condition is equivalent to that $\flR_{\frac{r}s}(q)$ is palindromic, and the latter condition is equivalent to $s^2 \equiv -1 \pmod{r}$. (Note that $0< s,t <r$ now.)
  
\end{proof}

\begin{thm}\label{palin}
The following hold. 
\begin{itemize}
\item[(1)] For an irreducible fraction $\frac{r}s$, $\flS_{\frac{r}s}(q)$ is palindromic 
if and only if $r^2 \equiv -1 \pmod{s}$; 
\item[(2)] For an irreducible fraction $\frac{r}s$, $\flR_{\frac{r}s}(q)$ is equivalent to a palindromic polynomial  if and only if $s^2 \equiv -1 \pmod{r}$;
\item[(3)] For an irreducible fraction $\frac{r}{s}>1$, the normalized Jones polynomial $J_{\frac{r}s}(q)$ is palindromic if and only if $s^2 \equiv -1 \pmod{r}$.
\end{itemize}
\end{thm}
\begin{proof}
(1) By Proposition~\ref{basic for flat} (1), we may assume that $0 < \frac{r}s, \frac{r'}s < 1$.  Now the assertion follows from Lemma~\ref{Palin for flR} and Proposition~\ref{basic for flat} (3). 

(2) The assertion follows from (1) and  Proposition~\ref{basic for flat} (3).

(3) The assertion follows from Theorem~\ref{J-L} and Lemma~\ref{Palin for flR}.
\end{proof}

Let $\frac{r}s=[a_1, \ldots, a_{2m}]>1$ be an irreducible fraction with 
\begin{equation}\label{M & A}
M:=M_q(a_1, \ldots, a_{2m})=\begin{pmatrix}
r & t \\
s &  u\\
\end{pmatrix} \quad \text{and} \quad 
A:=(-S_q)M. 
\end{equation}

By our proof of Theorem~\ref{Tr A}, we have 
\begin{equation}\label{trace and Jones}
(q-1)\Tr A = \flR_{\frac{r}s}(q)-\flR_{\frac{r}s}(q)^\vee=J_{\frac{r}s}(q)^\vee-J_{\frac{r}s}(q).
\end{equation}
For an irreducible fraction $\alpha=\frac{r}s >1$ with $s^2 \not \equiv -1 \pmod{r}$, $J_\alpha(q)$ is not palindromic by Theorem~\ref{palin} (3), and $J_\alpha(q)^\vee-J_\alpha(q)$ is equivalent to a non-zero anti-palindromic polynomial which can be divided by $(1-q)$. So there is a unique palindromic polynomial $I_\alpha(q) \in \ZZ[q]$ with $I_\alpha(0) > 0$ and 
$$I_\alpha(q) \equiv \frac{J_\alpha(q)^\vee-J_\alpha(q)}{1-q}.$$

If $J_\alpha(q)$ is palindromic, we set $I_\alpha(q)=0$.  
By \eqref{trace and Jones}, we have $I_\alpha(q) \equiv \Tr A$. 

\begin{cor}
With the above notation, we have $I_\alpha(q) \in \ZZ_{\ge 0}[q]$ and $I_\alpha(0)=1$.   
\end{cor}

\begin{proof}
By \eqref{trace and Jones}, the assertion follows from Proposition~\ref{3 types for trace}, 
while it remains to show that $I_\alpha(q) \ne 2$. 
By the proof of Theorem~\ref{Tr A}, we see that if $\Tr A \equiv 2$ then $A$ is equivalent to the identity matrix. In this case, we have $M \equiv S_q$, but it is not possible since $M=M_q(a_1, \ldots, a_{2m})$  for $a_1, \ldots, a_{2m} >0$. \end{proof}

By Proposition~\ref{3 types for trace},  Conjecture~\ref{Oguz} implies the following conjecture on the normalized Jones polynomials of rational links (equivalently, left $q$-deformed rationals).

\begin{conj}
$I_\alpha(q)$ is unimodal except for the following two types.  
\begin{itemize}
\item[(1)]  $1+q^n$ for some $n \ge 2$. 
\item[(2)] $\sum_{i=0}^{2k}a_iq^i$ with 
$(a_0, \ldots, a_{2k+1})=(1,2,\ldots, k, k-1, k, k-1, k-2, \ldots, 2,1)$ for some $k \ge 2$.  
\end{itemize}
Especially, $I_\alpha(q)$ is at most bimodal. 
\end{conj}

The case (2) (resp. (3)) corresponds to the case 
$$A \equiv R^n \quad  \text{( resp. $A=M_q(1,k,1,k), M_q(k,1,k,1)$ )}.$$ Note that $\Tr M_q(1,k,1,k) (=\Tr M_q(k,1,k,1))$ is computed in  
\cite[Example 1.8.]{OR}.  

\begin{exam}

Easy calculation shows that $I_{\frac{12}{5}}(q)$ and $I_{\frac{15}{4}}(q)$ are not unimodal. In fact, 
$I_{\frac{12}{5}}(q)=I_{\frac{12}{7}}(q)=q^{3}+1$ and 
 $I_{\frac{15}{4}}(q)=I_{\frac{15}{11}}(q)=q^{4}+2 q^{3}+q^{2}+2 q^{}+1.$

\end{exam}

Finally, we give the following proposition.

\begin{prop}
For any irreducible fraction $\alpha >1$, there is an infinite sequence of irreducible fractions $\alpha_1, \alpha_2, \ldots $$($every $\alpha_i>1$ are distinct$)$ such that $I_{\alpha_i}(q)=I_{\alpha}(q)$ for all $i \ge 1$.
\end{prop}

\begin{proof}
    With the same notation as \eqref{M & A},
we have $$I_\alpha(q)
\equiv \shS_{\frac{r}s}(q) -\shS_{\frac{r}t}(q)=\shR_{\frac{s}u}(q)-\shR_{\frac{t}u}(q)$$
by \eqref{s-t}. For $\alpha_1:= \dfrac{r+s+t+u}{s+u}=[b_1, \ldots, b_{2l}]>1$, we have 
$$M(b_1, \ldots, b_{2l})= \begin{pmatrix}
r+s+t+u & t+u \\
s+u &  u\\
\end{pmatrix}$$
and 
$$I_{\alpha_1}(q)\equiv \shR_{\frac{s+u}u}(q)-\shR_{\frac{t+u}u}(q)\equiv \shR_{\frac{s}u}(q)-\shR_{\frac{t}u}(q)\equiv I_\alpha(q).$$
Hence $I_{\alpha_1}(q)=I_\alpha(q)$. 
Repeating this procedure, we get the sequence $\{ \alpha_i \}_{i \ge 1}$ with 
$$\alpha_i=\frac{r+i(s+t)+i^2u}{s+iu},$$
which satisfies the desired property.  
\end{proof}



\section*{Acknowledgments}

We would like to thank Professor Michihisa Wakui for helpful comments. 
The first author is supported by JSPS KAKENHI Grant Number 21H04994, 19K03456. The second author is supported by JSPS KAKENHI Grant Number 22K03258.

\end{document}